\documentclass[12pt,reqno]{amsart}

\newcommand{\phip}{\phi_{P^o} }
\newcommand{\QQ}{\mathcal{Q}}

\newcommand{\T}{{\mathbf T}^m}

\newcommand{\szego}{Szeg\"o }

\newcommand{\kahler}{K\"ahler }

\newcommand{\PP}{{\mathbb P}}
\newcommand{\R}{{\mathbb R}}
\newcommand{\C}{{\mathbb C}}
\newcommand{\Q}{{\mathbb Q}}
\newcommand{\Z}{{\mathbb Z}}

\newcommand{\N}{{\mathbb N}}
\newcommand{\CP}{\C\PP}

\newcommand{\dbar}{\bar\partial}
\newcommand{\ddbar}{\partial\dbar}

\newcommand{\E}{{\mathbf E}\,}

\renewcommand{\phi}{\varphi}

\newcommand{\bcal}{\mathcal{B}}
\newcommand{\ccal}{\mathcal{C}}
\newcommand{\dcal}{\mathcal{D}}
\newcommand{\ecal}{\mathcal{E}}

\newcommand{\gcal}{\mathcal{G}}
\newcommand{\hcal}{\mathcal{H}}

\newcommand{\lcal}{\mathcal{L}}

\newcommand{\pcal}{\mathcal{P}}

\newcommand{\ocal}{\mathcal{O}}

\def    \Z  {{\mathbb Z}}
\def    \R  {{\mathbb R}}
\def    \C  {{\mathbb C}}

\newtheorem{theo}{{\sc Theorem}}[section]
\newtheorem{cor}[theo]{{\sc Corollary}}

\newtheorem{lem}[theo]{{\sc Lemma}}
\newtheorem{prop}[theo]{{\sc Proposition}}

\newenvironment{rem}{\medskip\noindent{\it Remark:\/} }{\medskip}
\newenvironment{defin}{\medskip\noindent{\it Definition:\/} }{\medskip}

\title[Test configurations, large deviations  and geodesic rays  on  toric varieties]
{Test configurations, large deviations  and geodesic rays on toric
varieties }

\author{Jian Song and Steve Zelditch }
\address{Department of Mathematics \\ Rutgers University, Piscataway, NJ 08854, USA}
\email{jiansong@math.rutgers.edu}
\address{Department of Mathematics, Johns Hopkins University, Baltimore,
MD 21218, USA} \email{zelditch@math.jhu.edu}

\thanks{Research partially supported by
 NSF grants DMS-0604805 and
 DMS-0603850.}

\date{\today}

\begin{document}

\begin{abstract}

This article contains  a detailed study in the case of a toric variety of the geodesic rays $\phi_t$  defined by Phong-Sturm corresponding to
 test configurations $T$ in the sense of Donaldson. We show that the `Bergman approximations' $\phi_k(t,z)$
  of Phong-Sturm converge in $C^1$ to the geodesic ray $\phi_t$, and that
the geodesic ray itself is $C^{1,1}$ and no better. In particular, the \kahler metrics $\omega_t = \omega_0 + i \ddbar \phi_t$
 associated to the  geodesic ray
of potentials  are discontinuous
across certain hypersurfaces and are degenerate on certain open sets.

A novelty in the analysis is the connection between Bergman metrics, Bergman kernels  and the theory of large deviations.
We construct  a sequence of measures $\mu_k^z$  on the polytope of
the toric variety, show that they satisfy a large deviations principle, and relate the rate function to the geodesic ray.

\end{abstract}

\maketitle

\tableofcontents

\section{Introduction}

This article is inspired by recent work of Phong-Sturm \cite{PS2}
on {\it test configurations} and geodesic rays for ample line
bundles $L \to M$ over \kahler manifolds $(M, \omega)$.  The main
construction in \cite{PS2} associates to a test configuration $T =
(\lcal \to \chi \to \C)$ and a hermitian metric $h_0$ on $L$ an
infinite geodesic ray  $R(h_0, T) = e^{- \psi_t} h_0$  starting at
$h_0$ in the infinite dimensional symmetric space $\hcal$ of
hermitian metrics on $L$ in a fixed \kahler class in the sense of
Mabuchi, Semmes and Donaldson \cite{M,S,D1}. At this time, it
seems to be the only known construction of an infinite geodesic
ray with given initial point in the incomplete space $\hcal$ (see
also \cite{AT} for other constructions of geodesic rays). The
 geodesic rays in \cite{PS2} are constructed by taking limits of `Bergman
 geodesic rays', i.e. geodesic rays  in the  finite dimensional symmetric spaces $\bcal_k$ of
 Bergman metrics. The purpose of this paper is to analyze the
 construction in detail in the case of  toric test
configurations on toric varieties. We give  explicit formulae
for the Bergman geodesic rays and for  the limit geodesic ray. The formulae   show
clearly  that the
geodesic rays produced by toric  test configurations are $C^{1,1}$
and not $C^2$, and that the approximating Bergman geodesics
converge in $C^1 ([0, L] \times M)$ for any $L > 0$. Furthermore,
the metrics $\omega_t = \omega_0 + i \ddbar \psi_t$ are only
semi-positive for $t > 0$, i.e. $\omega_t^m = 0$ on certain open
sets (cf. Theorem \ref{GEO}). Hence, both in terms of regularity
and positivity, the rays lie in some sense on the boundary of
$\hcal$, and we obtain a weak solution of the Monge-Amp\`ere
equation which saturates the known $C^{1,1}$ regularity results
\cite{Ch}.

Two  other articles have recently appeared which contain regularity results on
test configuration geodesic rays. First,
 $C^{1,1}$  geodesic rays are constructed by  Phong-Sturm \cite{PS3} from
  any test configuration using a resolution of singularities. At least when  the total space of the test configuration
  is smooth, the rays must be the limits of  Bergman rays. In our toric setting, the total space is never smooth and
  it is not clear at present how our regularity results overlap.
 Second, the $C^{1,1}$ regularity of test configuration geodesic
 rays with smooth total space was also observed in the  article \cite{CT} of  Chen-Tang.
  The authors also  give examples of toric test configuration geodesic rays which  are not smooth \cite{CT}.

In proving the convergence result, we employ a novel connection
between analysis on toric varieties  and  the theory of large
deviations. As we will show in Theorems \ref{GEO} and \ref{LDT},
the Phong-Sturm geodesic ray $\psi_t$ arises from Varadhan's Lemma
applied to a family $\{\mu_k^z\}$ of probability measures on the
polytope $P$ of $M$ which are defined by  the test configuration.
It  is  closely related to the rate functional of a large
deviations principle for another `time-tilted' family
$\mu_k^{z,t}$ of probability measures on $P$. We believe that this
connection is of independent interest and therefore develop it in
some depth for its own sake. The measures $\mu_k^z$ are closely
related to  the random variables and probability measures on $\R$
defined in \cite{PS} for the geodesic problem on a general \kahler
variety (see the first remark in Section 5). To be precise, the
Phong-Sturm measures are the pushforwards to $\R$ under a certain
function of the $\mu_k^z$. In subsequent  work, we study the
asymptotic properties and large deviations properties of the $\R$
measures on any \kahler manifold.

To state our results, we need some notation.  Let $M$ be a smooth
$m$-dimensional  toric variety, let $\T$ denote the real torus
$(S^1)^m$ which acts on $M$, and let $L \to M$ be a very ample
toric line bundle. Let $h_0 \in \hcal$ be a positively curved
reference metric on $L$ in the give \kahler class, let $\omega_0$
denote its curvature $(1,1)$ form and let $\mu_0$ denote the
moment map for the $\T$ action on $M$ with respect to $\mu_0$.  We
denote by  $T$ a test configuration. In the case of a toric
variety, Donaldson \cite{D1} shows that general toric test
configurations are determined by rational piecewise linear convex
functions
\begin{equation} f = \max\{\lambda_1, \dots, \lambda_p\}, \;\;
\mbox{with} \;\;\lambda_j (x) = \langle \nu_j, x\rangle + v_j
\end{equation}
on the polytope $P$ of the toric variety, where  $\lambda_j (x)$
are affine-linear functions
 with rational
coefficients. Roughly speaking, the graph of $R - f(x)$ for a
large integer $R > 0$ is the `top' of an $m + 1$-dimensional
polytope $Q$ with base the $m$-dimensional polytope $P$ and the
degeneration occurs as one moves from the bottom to the top. By
multiplying $f$ by $d$ we may assume the affine functions
$\lambda_j$ have integral coefficients. We denote by $P_j \subset
P$ the subdomain where $f = \lambda_j$.

In \cite{PS2}, the geodesic ray $e^{- \psi_t} h_0$ is constructed
as a limit of  Bergman geodesic rays $h(t; k)= h_0 e^{-
\psi_k(t,z)}$ which are constructed from the test configuaration
$T$ (see Definition \ref{PSGR}). In the following Proposition, we
given an exact formula for $\psi_k$ in the case of a toric test
configuration. It is stated in terms of monomial sections
corresponding to lattice points in the polytope of $M$.  We refer
to \S \ref{BACKGROUND} for as yet undefined terminology.

\begin{prop} \label{INTROPROP} Let $(M, L, h_0, \omega_0)$ be a polarized toric \kahler
variety of dimension $m$,
  and let $P$ denote the
corresponding lattice polytope. Then the Phong-Sturm sequence of
approximating Bergman geodesics is given by
\begin{equation} \label{toricphik} \psi_k(t, z) = \frac{1}{kd} \log
\hat{Z}_k(t, z)
\end{equation}
with
\begin{equation} \label{toricFk}\left\{ \begin{array}{l}  \hat{Z}_k(t, z) =  e^{- 2 t
\frac{1}{d_k} \sum_{\alpha \in k d P \cap \Z^m} k d( R -
f(\frac{\alpha}{kd}))} Z_k^{z,t}, \\ \\
Z_k^{t,z} : = \sum_{\alpha \in k d P \cap \Z^m} e^{2 t kd ( R -
f(\frac{\alpha}{kd}))} \frac{||s_{\alpha}(z)||^2_{h_0^{
kd}}}{\QQ_{h_0^{k d}}(\alpha)}, \end{array} \right.
\end{equation}
where $\{s_{\alpha}\}$ is the basis of $H^0(M, L^{kd})$ corresponding
to the monomials $z^{\alpha}$ on $\C^m$ with $\alpha \in k d P$, $d_k = H^0(X, L^{kd})$,
and where $\QQ_{h_0^{k d}}(\alpha)$ is the square of its $L^2$
norm with respect to the inner product $Hilb_{kd}(h_0)$ induced by
$h_0$ (see (\ref{HILB}), (\ref{QFORM}) and (\ref{SPNORM}) for the
precise formula).

\end{prop}

The Phong-Sturm geodesic ray is by definition   the  limit
$\psi_t(z)$ (in a certain topology) of the sequence $\psi_k(t,
z)$. Our next result gives an explicit formula for it.  One of the
main points of this article is that this relative \kahler
potential is naturally expressed in terms of the {\it rate
functions} $I^z$ for the large deviations principle of the
sequence of probability measures,
\begin{equation} \label{MUKZDEF} \mu_k^z = \frac{1}{\Pi_{h_0^{k d}}(z,z)}\;\; \sum_{\alpha \in kd P \cap \Z^m}
\frac{|s_{\alpha}(z)|_{h_0^{k d}}^2}{||s_{\alpha}||_{h_0^{k d}}^2}   \;
\delta_{\frac{\alpha}{k d}},
\end{equation} where $\Pi_{h_0^{k d}}(z,z)$ is the contracted \szego kernel on the
diagonal (or density of states); see \S \ref{BACKGROUND} for
background. We obviously have
\begin{equation} \label{PHIKTZINT} \psi_k(t,z) =  \frac{1}{k d} \log  \int_P e^{k dt (R - f(x))}
d\mu_k^z(x) +  2 t \frac{1}{d_k} \sum_{\alpha \in k d P \cap \Z^m}
k d( R - f(\frac{\alpha}{kd})),
\end{equation}
 and since the second term has an
obvious limit, the determination of $\psi_t$ reduces  to  the
uniform asymptotics of the first term. For notational simplicity
we henceforth often write
\begin{equation} \label{FT}
F_t(x) = t (R - f(x)). \end{equation}

Our first observation is the following measure concentration
result, which follows  easily from the Bernstein polynomial
results of \cite{Z2}.
\begin{prop} \label{LLNMUZ} Let $\mu_0: M \to P$ be the moment map
with respect to the symplectic form $\omega_0.$  Then  for any $z
\in M$,  the measures $\mu_k^z$ tend weakly to
$\delta_{\mu_0(z)}$. Thus,
$$\mu_0(z) = \lim_{k \to \infty} \frac{1}{\Pi_{h^{kd}}(z,z)} \sum_{\alpha \in k d P}
 (\frac{\alpha}{k d}) \;\; \frac{||s_{\alpha}(z)||^2_{h_0^{k d}}}{\QQ_{h_0^{k d}}(\alpha)}. $$
\end{prop}

A much deeper result is that $\mu_k^z$ satisfy a large deviations
principle (LDP), and that the   logarithmic asymptotics in
(\ref{PHIKTZINT}) are therefore  determined by Varadhan's Lemma.
Heuristically,  an LDP means that the measure $\mu_k^z(A)$ of a
Borel set $A$ is obtained asymptotically by integrating $e^{- k
I^z(x)}$ over $A$, where $I^z$ is known as the rate functional and
$k$ is the rate.
 The rate functions $I^z$ for $\{d\mu_k^z\}$ depend on whether $z$ lies
 in the open orbit $M^o$ of $M$ or on the divisor at infinity $\dcal$;
 equivalently, they depend on whether the image $\mu_0(z)$ of $z$ under the
 moment map for $\omega_0$ lies in the  interior $P^o$
 of the polytope $P$ or  along a face $F$ of its boundary
 $\partial P$. The definition of uniformity of the Laplace large deviations
 principle will be given in section \S \ref{ULP}. The notation and terminology will be defined and
 reviewed in  \S \ref{BACKGROUND}.

\begin{theo}\label{LDINTRO}  For any  $z \in M$,  the probability measures
$\mu_k^z$  satisfy a uniform Laplace  large deviations principle
with rate $k$ and with convex rate functions $I^z \geq 0$ on $P$
defined as follows:

\begin{itemize}

\item If $z \in M^0$, the open orbit, then  $I^z(x) = u_0(x) -
\langle x, \log |z| \rangle + \phip (z),$  where $\phip$ is the
canonical \kahler potential of the open orbit and $u_0$ is its
Legendre transform, the  symplectic potential;

\item When $z \in \mu_0^{-1}(F)$ for some face $F$ of $\partial
P$, then $I^z(x)$ restricted to $x \in F$ is given by
 $I^z(x) =  u_F(x) - \langle x', \log |z'| \rangle + \phi_F(z),$
 where $\log |z'|$ are orbit coordinates along $F$,  $\phi_F$ is
 the canonical
 \kahler potential for the subtoric variety defined by $F$ and $u_F$ is its Legendre transform.
  On the
 complement of $\bar{F}$ it is defined to be $+\infty$.

 \item When $z $ is a fixed point then $I^z(v) = 0$ and elsewhere
 $I^z(x) = \infty$.

\end{itemize}
 \end{theo}

The large deviations principle seems to us of independent
interest. The  statement for fixed $z$  follows from the
G\"artner-Ellis theorem (cf. \cite{DZ,dH,E}), once it is
established that the hypotheses of this theorem are satisfied.  In
addition, we prove that the upper and lower bounds are in a
certain sense  uniform in the $z$ parameter. This is a rather
complicated matter in our setting, since the rate functions $I^z$
are highly non-uniform, and we follow the definition of Laplace
large deviations principles of \cite{DE} in defining uniform large
deviations. They are precisely adapted to Varadhan's Lemma and
will imply that $\phi_k \to \phi$ in $C^0$. From this large
deviations result and a more substantial one below (Theorem
\ref{LDT}), we obtain our convergence result:

\begin{theo} \label{GEO} Let $L \to X$ be a very ample toric line bundle over
a toric \kahler manifold. Let $h_0$ be a positively curved metric
on $L$ and $T$ a test configuration. Then $\psi_k(t; z)$ converges
in $C^1$ to the  $C^{1,1}$ geodesic ray $\psi_t(z)$  in $\hcal$
given by (cf.  (\ref{FT}))
\begin{equation} \label{PHIT} \psi_t (z) = \sup_{x \in P} [F_t - I^z]  \end{equation}
Moreover, $\psi_t$
 has a bounded but discontinuous second spatial and $t$ derivatives and
 $\omega_t = \omega_0 + \ddbar \psi_t$ has a zero eigenvalue in certain open sets.
 Explicitly
\begin{itemize}

\item  When $z \in M^o$,  the open orbit,
 $\psi_t(z) = \lcal_{\R^m} (u_0 + t f) - \phip $ where $\lcal_{\R^m}$ is the Legendre
 transform on $\R^m$. There is a simpler formula which uses the
 moment map $\mu_t$ associated to $\omega_t$ which is introduced in
 Proposition \ref{LLNMUZT}. As will be shown in Proposition \ref{PHITFORM}, in
 the region $\mu_t^{-1}(P_j)$, we
 have $\psi_t(e^{\rho/2 + i \theta}) = \phip ( e^{(\rho + t \nu_j)/2+ i\theta}) - t v_j - \phip$, where $z= e^{\rho/2+\i \theta}$.

 \item When $z \in \mu_0^{-1}(F^o)$, then $\psi_t(z) = \lcal_{F^o} (u_F + t
 f) - \phi_F$ where $\lcal_{F^o}$ denote the Legendre transform on the
 quotient of $\R^m$ by the isotropy subgroup of $z$;

 \item When $z = \mu_0^{-1}(v)$ is a vertex, then $\psi_t(z) = - t f(v).$

 \item  A point  $z = e^{\rho/2+ i\theta} \in \mu_0^{-1}(P_j \cap P_k)$ only if
 $\rho \in t CH(\nu_j, \nu_k)$, where $CH$ denotes the convex hull.
  In that case,  $\mu_t(\rho) $ is a constant point $x_0$
 for $\rho \in t CH (\nu_j, \nu_k)$, and
$\psi_t(z) = \langle \rho, x_0 \rangle - u_0(x_0) - t (\langle
\nu_j, x_0 \rangle + v_j) - \phip(z). $ In these open sets,
$\omega_t^m \equiv 0.$ Analogous formulae hold on the faces of
$\partial P$.

 \end{itemize}
 \end{theo}

We obtain the formula for $\psi_t$ by applying Varadhan's Lemma to
the integrals (\ref{PHIKTZINT}). Uniformity of the limit is a
novel feature. An additional  part of the proof of Theorem
\ref{GEO} is to show that (\ref{PHIT}) is in fact a $C^{1,1}$
function on $M$. Even on the open orbit, it requires some  convex
analysis to see that $\lcal(u_0 + t f) \in C^{1,1}(\R^m)$. Roughly
speaking, the Legendre transforms smooths out the corners of $f$
to $C^1$, but no further than $C^{1,1}$. We must then verify that
the extension to $M$ of $\psi_t$ remains $C^{1,1}$. Since $\psi_t$
is $C^1$ it determines a moment map, given over the open orbit by
\begin{equation} \label{MUPHI} \mu_t: M^o \to P, \;\; \mu_t(e^{\rho/2 + i \theta}) =
\nabla_{\rho} \psi_t (e^{\rho/2 + i \theta}) \;\; \mbox{on} \;\; M^o.
\end{equation}

 An interesting feature of the convex analysis is that
$\mu_t$ fails to be a homeomorphism from $M/\T$ to $P$ as in the
smooth case.
 Indeed, the usual inverse map defined by gradient of the symplectic
potential pulls apart the polytope discontinuously into different
regions. However,  in  \S \ref{MMSUBDIF} we explain how the moment
map of the singular metrics $\psi_t$ is rather a homeomorphism
from the underlying real toric variety $M_{\R}$ (cf. \S
\ref{BACKGROUND})  to the graph of the subdifferential of $u + t
f$. More precisely, this statement is correct on the open orbit
and has a natural closure on the boundary.

The proof of $C^0$ convergence to the limit is already an
improvement on the degree of convergence in \cite{PS2}.  To prove
$C^1$ convergence, we consider integrals  of bounded continuous
functions on $P$ against the `time-tilted' probability measures
\begin{equation} \label{TIMETILT} d\mu_k^{z,t}: = \frac{1}{Z_k^{t, z}} \sum_{\alpha \in k d P \cap \Z^m} e^{2  k d
(F_t(\frac{\alpha}{k d}))} \frac{||s_{\alpha}(z)||^2_{h_0^{k d}}}{\QQ_{h_0^{k d}}(\alpha)} \delta_{\frac{\alpha}{k}}.
\end{equation}

It follows from  the `tilted' Varadhan Lemma and from Theorem
\ref{LDINTRO} that  the sequence
 $d\mu_k^{z,t}$ of probability measures on $P$  satisfies a  large deviations principle with convex rate function
$$I^{z,t} = I^z + t f(x)  + \sup_{x \in P} [F_t - I^z]. $$ On
the open orbit, $I^{z, t} (x) = - \langle x, \log |z| \rangle +
u_0(x) + \phip (\rho)$. Thus, the \kahler potential of the
singular metric corresponding to $u + t f$, given in Theorem
\ref{GEO} (see (\ref{PHIT})), is closely related to the rate
functional for $\mu_k^{z, t}$.

 The importance of these tilted measures is that the
first derivatives $d \psi_k(t,z)$ in $t$ and $z$ variables can be
expressed as integrals of  continuous functions on the compact set
$P$  against $d\mu_k^{t, z}$. We  prove the following  parallel to
Proposition \ref{LLNMUZ}:

\begin{prop} \label{LLNMUZT}  For all $(z,t)$, there exists a unique limit point
$\mu_t(z) \in P$
$$\mu_t(z) =  \lim_{k \to \infty} \int_P x d\mu_k^{z,t}, $$
which is the same as (\ref{MUPHI}).  We have  $\mu_k^{z,t} \to
\delta_{\mu_t(z)}$ in the weak sense as $k \to \infty.$ Further,
$d_{z, t} \psi_k(t,z) \to d_{z, t} \psi(t,z)$ uniformly as $k \to
\infty$. More precisely,
$$\left\{ \begin{array}{l} d \psi_k(t, z) \to \bar{x}(t, z) = \lim_{k \to \infty} \int_P x d\mu_k^{z, t},\\ \\
\dot{\psi}_k(t,z)  \to R - f(\bar{x}(t,z)) = \lim_{k \to \infty}
\int_P (R - f(x))d\mu_k^{z,t}.
\end{array} \right.$$

 \end{prop}

The organization of this article is as follows: After reviewing
the basic geometric and analytic objects in \S \ref{BACKGROUND},
we introduce the measures $d\mu_k^z$ in \S \ref{LDMUKZ} and prove
Proposition \ref{LLNMUZ} and state Theorem \ref{LDINTRO}.  In \S
\ref{LOGMGFl} we begin the proof of Theorems \ref{LDINTRO}  by
studying the scaling limit of the logarithmic moment generating
function of $d\mu_k^z$. In \S \ref{PROOFLD} we prove Theorem
\ref{LDINTRO}. In \S \ref{TTC}, we introduce toric test
configurations and prove Proposition \ref{INTROPROP}. In \S
\ref{ULP}, we prove the $C^0$ convergence part of Theorem
\ref{GEO} using  Varadhan's Lemma (Theorem \ref{C0}). This proves
the explicit formula (\ref{PHIT}). We then study the regularity of
$\psi_t$ in \S \ref{MMSUBDIF}, define the moment map $\mu_t$ and
study its regularity and mapping properties. The analysis shows
that $\psi_t \in C^1(M)$ for each $t$. In \S \ref{PFC1} we show
that $\psi_k \to \psi$ in $C^1([0, L] \times M)$ for any $L > 0$.
We cannot of course obtain convergence in $C^2$ since $\psi_t
\notin C^2$.


In conclusion, we  would like to thank  D.H. Phong and J. Sturm
for many discussions of the subject of this article as it evolved.
We also thank V. Alexeev for tutorials on toric test
configurations and in particular for corroborating Proposition
\ref{PROPWEIGHTS}. We thank  O. Zeitouni for tutorials  in the theory of
large deviations. The initial results of this paper were presented
at the Complex Analysis and Geometry – XVIII conference in  Levico
in May,  2007.

\section{\label{BACKGROUND} Background on toric varieties}

We employ the same notation and terminology as in
\cite{D1,SoZ1,SoZ2}. We briefly recall the main definitions for
the reader's convenience.

 We recall that a toric \kahler manifold is a \kahler manifold
$(M, J, \omega_0)$ on which the complex torus $(\C^*)^m$ acts
holomorphically with an open orbit $M^o$.  We
 choose a basepoint $m_0$ on the orbit open and
identify  $M^o \equiv (\C^{*})^{m}$. The underlying real torus is
denoted $\T$ so that $(\C^*)^m = \T \times \R_+^m$, which we write
in coordinates as $z = e^{\rho/2 + i \theta}$ in a multi-index
notation.

We fix the standard basis $\{\frac{\partial}{\partial \theta_j}\}$
of $Lie(\T)$ (the Lie algebra) and use the same notation for the
induced vector fields on $M$. We also denote by
$\{\frac{\partial}{\partial \rho_j}\}$ the standard basis of
$Lie(\R_+^m)$. We then have $\frac{\partial}{\partial \rho_j} =
J_0 \frac{\partial}{\partial \theta_j}$ where $J_0$ is the
standard complex structure on $\C^m$. We use the same notation for
the induced vector fields on $M$.

We assume that $M$ is projective and that $P$ is a Delzant
polytope; we view $M$ as defined by a monomial embedding. The
polytope $P$ is
 defined by a set of linear inequalities
$$l_r(x): =\langle x, v_r\rangle-\alpha_r \geq 0, ~~~r=1, ..., d, $$
where $v_r$ is a primitive element of the lattice and
inward-pointing normal to the $r$-th $(n-1)$-dimensional face of
$P$. We denote by $P^o$ the interior of $P$ and by $\partial P$
its boundary; $P = P^o \cup
\partial P$. We normalize $P$ so that $0 \in P$ and $P \subset
\R_{\geq}^m$. Here, and henceforth, we put $\R_{\geq} = \R_+ \cup
\{0\}$.  For background, see \cite{G,A,D1,Fu}.

Underlying the complex toric variety $M$ is a real toric variety
$M_{\R}$, namely the closure of $\R_+^m$ under the monomial
embedding. It is a manifold with corners  homeomorphic to $P$ (cf.
\cite{Fu}, Ch. 4). Every metric moment map we consider below
defines such a homeomorphism.

\subsection{Monomial basis of $H^0(M, L^k)$, norms and \szego kernels}

Let $\# P $ denote the number of lattice points $\alpha \in \N^m
\cap P$. We denote by $L \to M$ the invariant line bundle obtained
by pulling back $\ocal(1) \to \CP^{\# P - 1}$ under the monomial
embedding defining $M$.
 A natural basis of the
space of holomorphic sections $H^0(M, L^k)$ associated to the
$k$th power of $L \to M$ is defined by the  monomials $z^{\alpha}$
where $\alpha$ is a lattice point in the $k$th dilate of the
polytope, $\alpha \in k P \cap \N^m.$ That is, there exists an
invariant frame $e$ over the open orbit so that $s_{\alpha}(z) =
z^{\alpha} e$. We denote the dimension of $H^0(M, L^k)$ by $N_k$.
We equip $L$ with a toric Hermitian metric $h = h_0$ whose
curvature $(1,1)$ form $\omega_0 = i \ddbar \log ||e||_{h_0}^2$
lies in $\hcal.$ We often express the norm in terms of a local
\kahler potential, $||e||_{h_0}^2 = e^{- \psi}$, so that
$|s_{\alpha}(z)|_{h_0^k}^2 = |z^{\alpha}|^2 e^{- k \psi (z)}$ for
$s_{\alpha} \in H^0(M, L^k)$.

Any  hermitian metric $h$ on $L$ induces inner products
$Hilb_k(h)$ on $H^0(M, L^k)$, defined by
\begin{equation} \label{HILB} \langle s_1, s_2 \rangle_{h^k} =
\int_M (s_1(z), s_2(z))_{h^k} \frac{\omega_h^m}{m!}.
\end{equation} The monomials are orthogonal with respect to any
such toric inner product and have the norm-squares
\begin{equation} \label{QFORM} Q_{h^k}(\alpha) = \int_{\C^m} |z^{\alpha}|^2 e^{-
k \psi(z)} dV_{\phi}(z), \end{equation} where $dV_{\phi} = (i
\ddbar \phi)^m/ m!$.

 The \szego (or Bergman) kernels of a positive Hermitian line
bundle $(L, h) \to (M, \omega)$  are the kernels of the orthogonal
projections $\Pi_{h^k}: L^2(M, L^k) \to H^0(M, L^k)$ onto the
spaces of holomorphic sections with respect to the inner product
$Hilb_k(h)$,
\begin{equation} \Pi_{h^k} s(z) = \int_M \Pi_{h^k}(z,w) \cdot s(w)
\frac{\omega_h^m}{m!}, \end{equation} where the $\cdot$ denotes
the $h$-hermitian inner product at $w$.
 In terms of a local frame  $e$  for $L \to M$ over an
open set $U \subset M$,  we may write sections as $s = f e$. If
$\{s^k_j=f_j e_L^{\otimes k}:j=1,\dots,d_k\}$ is  an orthonormal
basis for $H^0(M,L^k)$, then  the \szego kernel can be written in
the form
\begin{equation}\label{szego}  \Pi_{h^k}(z, w): = F_{h^k}
(z, w)\,e^{\otimes k}(z) \otimes\overline {e^{\otimes
k}(w)}\,,\end{equation} where
\begin{equation}\label{FN}F_{h^k}(z, w)=
\sum_{j=1}^{N_k}f_j(z) \overline{f_j(w)}\;, ~~~N_k = H^0(M, L^k).\end{equation} In the
case of a toric variety with $0 \in \bar{P}$, there exists a frame
$e$ such that $s_{\alpha}(z) = z^{\alpha} e$ on the open orbit,
and then
\begin{equation}\label{FNa}F_{h^k}(z, w)=
\sum_{j=1}^{N_k} \frac{z^{\alpha}
\bar{w}^{\alpha}}{Q_{h^k}(\alpha)} \;.\end{equation} Along the
diagonal, $\left(F_{h^k}(z,z)\right)^{-1}$ is a Hermitian metric.
The product $F_{h^k}(z,z) ||e||_{h^k}^2$ is then the ratio of two
Hermitian metrics and it balances out to have a power law
expansion,
\begin{equation} \label{TYZ}  \Pi_{h^k}(z,z) = \sum_{i=0}^{N_k}
||s^k_i(z)||_{h_k}^2 = a_0 k^m + a_1(z) k^{m-1} + a_2(z) k^{m-2} +
\dots \end{equation} where $a_0$ is constant;  see \cite{T,Z}. We
note that by a slight abuse of notation,  $\Pi_{h^k}(z,z)$ denotes
the metric contraction of (\ref{szego}). It is sometimes written
$B_{h^k}(z)$ and referred as the density of states. If we sift out
the $\alpha$th term of $\Pi_{h^k}$  by means of Fourier analysis
on $\T$, we obtain
\begin{equation} \label{PHK} \pcal_{h^k}(\alpha, z): =
\frac{|z^{\alpha}|^2 e^{- k \psi(z)}}{Q_{h^k}(\alpha)},
\end{equation}
which play an important role in this article (as in \cite{SoZ2}).

\subsection{\label{BACKGROUNDOO}\kahler potential on the open orbit and symplectic
potential}

On  any simply connected open set, a \kahler metric may be locally
expressed as $\omega = i \ddbar \phi$ where $\phi$ is a locally
defined function which is unique up to the addition $\phi \to \phi
+ f(z) + \bar{f}(\bar{z})$ of the real part of  a holomorphic or
anti-holomorphic function $f$. Of course, the potential is not
globally defined. We now introduce special local \kahler
potentials adapted to the open orbit, respectively the divisor at
infinity on a toric variety.

Without loss of generality, we assume that $L$ is very ample. Then on the open orbit $M^o \subset M$, there is a canonical choice of
the open-orbit \kahler potential once one fixes the image $P$ of
the moment map:
\begin{equation} \label{CANKP} \phip(z) : = \log  \sum_{\alpha \in
P}  |z^{\alpha}|^2 = \log  \sum_{\alpha \in P} e^{\langle \alpha,
\rho \rangle}.
\end{equation}
This is the potential appearing in Theorem \ref{LDINTRO} for the
open orbit. For instance, the Fubini-Study \kahler potential is
$\phi(z) =\log (1 + |z|^2) = \log (1 + e^{\rho})$. We observe
that, since $0 \in P$, (\ref{CANKP}) defines a  smooth function to
the full affine chart $z \in \C^m$ in the closure of the open
orbit chart when we use the first expression with  $z =
e^{\rho/2}$. As will be discussed in \S \ref{BACKGROUNDDP}, this
affine chart corresponds to the choice of the vertex $0$ and the
associated fixed point $\mu_0^{-1}(0)$. There is a corresponding
affine chart and \kahler potential on the chart for any vertex.

 Since it is
invariant under the real torus $\T$-action,  $\phip$ only depends
on the $\rho$-variables and we have
$$\omega_0 =  i  \sum_{j, k} \frac{\partial^2 \phip}{\partial
\rho_k \partial \rho_j} \frac{dz_j}{z_j} \wedge
\frac{d\bar{z}_k}{\bar{z}_k}.$$ Since $\omega_0$ is a positive
form, $\phip $ is a strictly convex function of $\rho \in \R^n$.
We may view
 $\phip (\rho)$  as a function on the Lie algebra
$Lie(\R_+^m)$ of $\R_+^m \subset (\C^*)^m$  or equivalently as a
function on the open orbit of the real toric variety $M_{\R}$.

The action of the real torus $\T$  on $(M, \omega_0)$ is
Hamiltonian with  moment map $\mu_0: M \to P$ with respect to
$\omega_0$.  We recall that the moment map $\mu_0: M \to (Lie
\T)^*$ is defined by $\langle \mu(z), X \rangle = H_X(z)$ where
$H_X$ is the Hamiltonian of $X^*$;  $X^*$ is the induced
Hamiltonian vector field  on $M$ induced by  natural map $X \in
Lie(\T) $.  Over the open orbit, the moment map may be expressed
as
\begin{equation} \label{MMDEF} \mu_{0} (z_1, \dots, z_m) =
(\frac{\partial \phip}{\partial \rho_1}, \cdots, \frac{\partial
\phip}{\partial \rho_m}) ;\;\; (z = e^{\rho/2 + i \theta}).
\end{equation}
Although the right side is an expression in terms of the locally
defined \kahler potential $\phip$, which is singular `at
infinity', the components $\frac{\partial \phip}{\partial \rho_j}$
extend to all of $M$ as smooth functions. This follows from the
fact that $\mu_0$ is globally smooth.  For instance, for the
Fubini-Study metric on $\CP^m$, the moment map is $\mu_0(z)  =
\frac{(|z_1|^2, |z_2|^2, \dots, |z_m|^2)}{1 + ||z||^2}, $ where
$||z||^2 = |z_1|^2 + \cdots + |z_m|^2$ and the \kahler potential
on the open orbit is $\phi_{P^o}(\rho) = \log (1 + e^{\rho_1 +
\cdots + \rho_m}),$ with $\frac{\partial \phip}{\partial \rho_j} =
\frac{e^{\rho_j}}{(1 + e^{\rho_1 + \cdots + \rho_m})}. $

The moment map defines a homeomorphism $\mu_0 : M_{\R} \to P$.
Later we will need to define the inverse of the map (\ref{MMDEF})
on $M_{\R}$  so we take some care at this point to make explicit
the identifications implicit in the formula.
 First,  we decompose  $Lie (C^*)^m =
Lie (\T) \oplus Lie (\R^m_+)$. Viewing $\phip$ as a function on
$Lie(\R_+^m)$,  $d \phip (\rho) \in T^*_{\rho} Lie(\R_+^m) \simeq
Lie(\R_+^m)^*$. Under $J_0: Lie(\R_+^m)^* \simeq Lie(\T)^*$ so we
may regard $d\phip: M_{\R} \to Lie(\T)^*$ as the moment map.

We now consider the {\it symplectic potential} $u_0$ associated to
$\phip$, defined as the  Legendre transform  of $\phip$ on $\R^m$:
\begin{equation} \label{SYMPOTDEF} u_{0}(x) = \phip^*(x) =  \lcal \phip(x): = \sup_{\rho \in \R^m}
(\langle x, \rho \rangle - \phip(e^{\rho/2 + i \theta})).
\end{equation}  It is a function on $P$, or in  invariant terms it is a function on   $
Lie(\T)^* \simeq Lie(\R_+^m)^*$. In general, the Legendre
transform of a function on a vector space $V$ is a function on the
dual space $V^*$.  The symplectic potential has canonical
logarithmic singularities on $\partial P$. According to
 \cite{A} (Proposition 2.8) or \cite{D1} ( Proposition 3.1.7),
 \begin{equation} \label{CANSYMPOT}
u_0(x) = \sum_k \ell_k(x) \log \ell_k(x) + f_0 \end{equation}
where $f_{0} \in C^{\infty}(\bar{P})$.

The differential $d u_0$ in a sense defines a partial inverse for
the moment map. More precisely, we have
 \begin{equation}\label{GRADULOGMU} (\frac{\partial u_0}{\partial x_1}, \dots, \frac{\partial u_0}{\partial x_m})
  = 2 \log \mu_0^{-1}
 (x) \iff \mu_0^{-1}(x) = \exp \frac{1}{2} (\frac{\partial u_0}{\partial x_1}, \dots, \frac{\partial u_0}{\partial
 x_m}),
 \end{equation}
 where $\exp: Lie(\R_+^m) \to \R_+^m$ is the exponential map.
In coordinates, this follows from the fact that
 $$u_0(x) = \langle x, \rho \rangle - \phip (e^{\rho/2+ i\theta}) \implies
 \nabla u_0(x) = \rho   - \langle x, \nabla_x \rho \rangle -
 \langle \nabla \phip (e^{\rho/2 + i \theta}), \nabla_x \rho \rangle = \rho , $$
when  $\nabla \phip(e^{\rho/2+ i\theta}) = x. $  To
 interpret
 (\ref{GRADULOGMU}) invariantly, we note that $d u_0(x) \in T^*_x (Lie (\T)^*)
 \simeq Lie(\T) \simeq Lie(\R_+^m)$ while
$\mu_0^{-1}
 (x) \in M^o_{\R} \simeq \R_+^m$ so that $2 \log \mu_0^{-1}
 (x) \in Lie (\R_+^m)$.
We observe that  (\ref{GRADULOGMU}) defines an inverse of $\mu_0$
from the open  orbit of the base point under $R_{+}^m$ to its
image $P^o$ and that it extends to a homeomorphism between the
manifolds with corners $\R_{\geq}^m \iff \mu_0(\R_{\geq}^m)$.
Here, $\R_{\geq} = R_+ \cup \{0\}$.

 It will also be important to write the norming constants in terms
 of the symplectic potential:
\begin{equation} \label{SPNORM} \QQ_{h^k}(\alpha) = \int_P e^{ k
(u_{0}(x) + \langle \frac{\alpha}{k} - x, \nabla u_{0}(x) \rangle}
dx. \end{equation} It follows from   \cite{SoZ2} (Proposition 3.1)
and from \cite{STZ}  that for interior $\alpha$, and $\alpha_k$
with $|\alpha - \alpha_k| = O(\frac{1}{k})$,
\begin{equation} \label{QQ} \QQ_{h^k}(\alpha_k) \sim k^{-m/2} e^{ k
u_0
(\alpha)}, \end{equation} and for all  $\alpha$ and $\alpha_k$
with $|\alpha - \alpha_k| = O(\frac{1}{k})$ that
\begin{equation} \label{QQa} \frac{1}{k} \log \QQ_{h^k}(\alpha_k) =
u_0
(\alpha) + O(\frac{\log k}{k}). \end{equation}


\subsection{\label{BACKGROUNDDP}The divisor at infinity $\dcal$ and the boundary $\partial P$ of $P$}

The above  definitions concern the behavior of the \kahler
potential on the open orbit $(\C^*)^m$ and the dual behavior of
the symplectic potential. As noted above, the \kahler potential
extends smoothly to the full affine chart $\C^m$. This is but one
affine chart needed to cover $M$ in the distinguished atlas
$\{U_{v}\}$ parameterized by vertices $v$ of $P$.  We briefly
explain how to modify the above constructions so that they apply
to the other charts, referring to \cite{SoZ1,STZ} for further
discussion.

For each vertex $v \in P$, we define  the chart $U_{v}$
 by
$ U_{v}:=\{z \in M\,;\,s_v(z) \neq 0\}, $ where $s_v$ is the
monomial section corresponding to $v$.  Since  $P$ is Delzant,
there exist  $\alpha^{1},\ldots,\alpha^{m} \subset \N^m \cap P$
such that each $\alpha^{j}$ lies on  an edge incident to  $v$, and
the vectors $v^{j}:=\alpha^{j}-v$ form a basis of $\Z^{m}$. We
define
\begin{equation}
\label{CChange} \eta_v:(\C^{*})^{m} \to (\C^{*})^{m}, \quad
\eta_v(z):=(z^{v^{1}},\ldots,z^{v^{m}}).
\end{equation}
The map $\eta$ is a $\T$-equivariant biholomorphism of $(\C^*)^m$
with inverse
\begin{equation}  z:(\C^{*})^{m} \to (\C^{*})^{m},\quad z(\eta)=(\eta^{\Gamma
e^{1}},\ldots,\eta^{\Gamma e^{m}}),
\end{equation}
where $e^{j}$ is the standard basis for $\C^{m}$, and $\Gamma$ is
 defined by
\begin{equation}\label{GAMMADEF}
\Gamma v^{j}=e^{j},\quad v^{j}=\alpha^{j}-v.
\end{equation}
 The corner of $P$ at $v$ is
transformed to the standard corner of the orthant $\R_+^m$ by
 the affine linear transformation
\begin{equation}\label{GAMMATWDEF}
\tilde{\Gamma}:\R^{m} \ni u \to \Gamma u -\Gamma v \in \R^{m},
\end{equation}
which preserves $\Z^{m}$, carries $P$ to a polytope $Q_{v} \subset
\{x \in \R^{m}\,;\,x_{j} \geq 0\}$ and carries the facets $F_j$
incident at $v$ to the coordinate hyperplanes $=\{x \in
Q_{v_0}\,;\,x_{j}=0\}$. The  map $\eta$ extends a homeomorphism: $
\eta:U_{v} \to \C^{m},$ and $$\eta(\mu_{P}^{-1}(\bar{F}_{j})) =
\{\eta \in \C^{m}\,;\,\eta_{j}=0\}.$$

For each $v$ we then define the \kahler potential $\phi_{U_v} $ on
$U_v \simeq \C^m$ by
\begin{equation} \label{PHIV} \phi_{U_v}(\eta) = \log \sum_{\alpha
\in P} |\eta^{\tilde{\Gamma}(\alpha)}|^2. \end{equation} The
Legendre transform of $\phi_{U_v}$ as a function on $\R^m$ defines
a dual symplectic potential $u_{U_v}$ on $P$. Generalizing
(\ref{GRADULOGMU}), the inverse of the moment map may be expressed
near the corner at $v$ by
\begin{equation}\label{GRADULOGMUv}  \mu_0^{-1}(x) =
 \exp \frac{1}{2} (\frac{\partial u_{U_v}}{\partial x_1}, \dots, \frac{\partial u_{U_v}}{\partial
 x_m}),
 \end{equation}
 where the right side is identified with a point in $U_v \cap
 M_{\R}$. Thus (\ref{GRADULOGMUv}) defines a homeomorphism from
 the corner at $v$ to its inverse image under $\mu_0$.

 To illustrate the notation in the simplest example of $\CP^1$
 with its Fubini-Study metric and with $v = 1$ we note that
 $\eta_v(z) = z^{-1}, \tilde{\Gamma}(u) = 1 - u$,
 $\phi_{U_1}(\eta) = \log (1 + |\eta|^2),$
and  $u_{U_1}(y) = y \log y + (1 - y) \log (1 - y)$. On  the
overlap $U_0 \cap U_1$,   we have $d u_{U_0}(x) =
 - du_{U_1}(y). $ Indeed,    $y = \tilde{\Gamma}(x)$ and so
 $d u (y) = \log \frac{y}{1 - y} = \log \frac{1 - x}{x} = -
du(x)$. Hence,  $e^{\frac{1}{2} d u(y)} = \frac{y}{1 - y} $ is the
inverse of $e^{\frac{1}{2} d u(x)}$ and $\mu_0^{-1}$ is locally
expressed as a map  from a neighborhood of $v = 1$  up to $y = 0$.

We also need to discuss moment maps and \kahler potentials for the
toric sub-varieties corresponding to boundary faces.  As in
\cite{Fu,SoZ1,STZ}, a face of $ P$ is the intersection of $ P$
with a supporting affine hyperplane;
 a top $m-1$-dimensional  face is  a facet; while at the
other extreme, the lowest dimensional faces are the vertices. We
denote the relative interior of a face by $F^o$.  Each  face
defines a sub-toric variety $M_F = \mu^{-1}(F) \subset \dcal$.
This subtoric variety also has an open orbit and a moment map. In
particular, over the open orbit $\mu^{-1}(F^o)$, there is  a
canonical \kahler potential for $\omega_0 |_{M_F}$:
\begin{equation} \label{CANKPF} \phi_F(z) = \log  \sum_{\alpha \in
\overline{F}}  |z^{\alpha}|^2 = \log  \sum_{\alpha \in
\overline{F}} e^{\langle \alpha, \rho'' \rangle},
\end{equation}
where now $\rho'' \in \R^{m - k}$ if $\dim \T_z = k$. Further the
Legendre transform of $\phi_F$ on $\R^{m-k}$ defines a symplectic
potential $u_F(x'')$ along $F$. Note that $u_0 = 0$ on $\partial
P$, so $u_F$ is not the restriction of $u_0$ to $F$. These
$\phi_F, u_F$ appear in Theorem \ref{LDINTRO} in the formula for
the rate function when $z \in F$. In the extreme case of a  vertex
$v$ corresponding to a fixed point of the $(\C^*)^m$,  $\phi_v =
0$.

\subsection{Summary of  \kahler potentials}

We summarize  the different notions of \kahler potential we have
introduced:

\begin{itemize}

\item $\psi_t$ is the relative \kahler potential with respect to
$h_0$  for the geodesic ray $h_t = e^{- \psi_t} h_0$. It is
globally defined on $M$ and $\psi_0 = 0$;

\item The Bergman geodesic ray potentials $\psi_k(t, z)$ (see
Definition \ref{PSGR}) are also   relative \kahler potentials,
with respect to $h_k$ arising from $Hilb_k(h_0)$. They are  also
globally defined on $M$ and are $o(1)$ as $k \to \infty$ at $t =
0$;

\item $\phip(z)$ is the open orbit \kahler potential corresponding
to $h_0$, i.e. it is the potential for $\omega_0$ on the open
orbit. Similarly, $\phi_F$ is the potential valid near
$\mu_0^{-1}(F)$.

\end{itemize}

\section{\label{TTC} Toric test configurations}

The purpose of this section is to give the proof of  Proposition
\ref{INTROPROP}. We include the basic definitions on toric test configurations
for the sake of completeness. Proposition
\ref{INTROPROP} is then simple to prove and is known to experts;  the
statement  can also be found in \cite{ZZ}.

 We first recall that a test configuration as
defined by Donaldson \cite{D1} consists of the following:

\begin{itemize}

\item A scheme $\chi$ with a $\C^*$ action $\rho$;

\item A $\C^*$ equivariant line bundle $\lcal \to \chi$ which is
ample on all fibers;

\item A flat $\C^*$ equivariant map $\pi: \chi \to \C$ where
$\C^*$ acts on $\C$ by multiplication.

\item The fiber $X_1$ is isomorphic to $X$ and $(X, L^r)$ is
isomorphic to $(X_1, L_1)$ where for $w \in \C, X_w = \pi^{-1}(w)$
and $L_w = \lcal|_{X_w}$.

\end{itemize}

For this article, only the {\it weights} $\eta_{\alpha}$ of the
test configuration  play a role, i.e. the weights of the $\C^*$
action on $H^0(X_0, L^k_0)$ where $X_0$ is the central fiber (i.e.
the eigenvalues of $B_k$ in the notation of \cite{PS2}). We define
the normalized weights (i.e. the eigenvalues of the traceless part
$A_k$ of $B_k$ in the notation of \cite{PS2})  by
\begin{equation} \label{NORMWEIGHTS} \lambda_{\alpha} =
\eta_{\eta} - \frac{1}{N_k} \sum_{j = 1}^{N_k} \eta_j
\end{equation}  The geodesic ray associated to the test
configuration is defined in terms of the normalized weights as
follows:

\begin{defin} \label{PSGR} The Phong-Sturm test configuration geodesic ray is
the weak limit of the Bergman geodesic rays  $h(t; k): (- \infty,
0) \to \hcal_k$ given by
$$h(t; k) = h_{\hat{\underline{s}}(t, k)} = h_0 e^{- \psi_k(t)},
$$
 with
$$\psi_k(t, z) = \frac{1}{k} \log \left( k^{-n} \sum_{\alpha =
0}^{N_k}  e^{2 t \lambda_{\alpha}}
\frac{|s_{\alpha}(z)|^2_{h_0^k}}{||s_{\alpha}||^2_{h_0^k}} \right)
$$
\end{defin}

\subsection{\label{INTROPROPROOF}Calculation of the weights}

In this section, we outline the calculation of the weights in the
case of a toric test configuration and prove Proposition
\ref{INTROPROP}

As above,  let  $P$ be the Delzant polytope corresponding to $M$,
and let $f: \R^m \to \R$ be the convex, rational piecewise-linear
function,
\begin{equation} f = \max\{\lambda_1, \dots, \lambda_p\},
\end{equation}
where the $\lambda_j$ are affine-linear functions with rational
coefficients.

 Fix an integer $R$ such that $f \geq R$ on $\bar{P}$
and following \cite{D1}, \S 4.2, we define a new polytope
\begin{equation} Q = Q_{f, R}  \subset \R^{m+1}, Q = \{(x, t): x
\in P, \; 0 < t < R - f(x)\}. \end{equation} By taking a multiple
$d Q$ it may be assumed that $Q$ is defined by integral equations.
Then $d Q$ is a  Delzant polytope of dimension $m + 1$  and
corresponds to  a toric variety $W$ of dimension $m + 1$ and a
line bundle $\lcal \to W$. When $t = 0$ we obtain a natural
embedding $\iota: (M, L^d) \to (W, \lcal)$. Intuitively,
the  toric degeneration is the singular toric
variety corresponding to the `top' of the polytope $k d Q$. It has
$p$ components,  one component for each facet of the top or
equivalently for  each of the affine functions $\lambda_j$
defining $f$. More precisely,

\begin{prop}\label{FUTAKI} (cf. \cite{D1} Proposition 4.2.1): There exists a
$\C^*$-equivariant map $p: W \to \CP^1$ so that $p^{-1} (\infty) =
\iota(M)$ such that $p$ restricted to $W \backslash \iota(M)$ is a
test configuration for $(M, L)$ with Futaki invariant
$$F_1 = - \frac{1}{2 Vol(P)} \left( \int_{
\partial P} f d\sigma -
\alpha \int_P f d\mu \right), \;\; \alpha = \frac{Vol(\partial
P)}{Vol(P)}. $$
\end{prop}

The map $p$ is defined as follows: For any $j$, the ratios $
\frac{s_{\alpha, j}}{s_{\alpha, j + 1}}$ define $\C^*$ equivariant
meromorphic functions on $W$. In fact, up to scale all of these
meromorphic functions agree. Hence we may define $p$ as the common
value of the ratios. The map is defined outside the common zeros
of $s_{\alpha, j}, s_{\alpha, j + 1}$. The sections for $j > 0$
all vanish on $\iota(M)$, so $p$ maps $\iota(M)$ to $\infty$.

The fibers $p^{-1}(t)$ are toric varieties isomorphic to $M$. The
central fiber is $p^{-1}(0) = M_0$. Then by definition, $M_0$ is
the zero locus of a holomorphic section $\sigma_0$ of
$p^*(\ocal(1))$. By using the exact sequences
$$0 \to H^0(W, \lcal^k(-1)) \to H^0(W, \lcal^k) \to H^0(M_0,
\lcal^k) \to 0$$ and
$$0 \to H^0(W, \lcal^k(-1)) \to H^0(W, \lcal^k) \to H^0(\iota(M),
\lcal^k) \to 0$$ defined  by multiplication by $\sigma_0$, resp.
$\sigma_1$,  one finds that
$$\dim H^0(M_0, \lcal^k) = \dim H^0(M, \lcal^k). $$

The prinicipal  fact we need about toric test
configurations is the following Proposition, which is implicit in \cite{ZZ} (Proposition 3.1):
\begin{prop} \label{PROPWEIGHTS} The  weights of the $\C^*$ action on the spaces
$H^0(M_0, \lcal^k)$ are given by
$$\eta_{\alpha} = k d \left(R - f(\frac{\alpha}{k d} ) \right), \;\; \alpha \in k d P
. $$
\end{prop}

\begin{proof}

The monomial basis of $H^0(W, \lcal^k)$
corresponds to lattice points in the associated lattice polytope
$k d Q$. The base of this polytope is thus $k d P$ and the height
over a point $x$ is $k d (R - f(\frac{x}{kd}))$.
The space $H^0(M_0, \lcal^k|_{M_0})$ is thus spanned
by the monomials $$z^{\alpha} w^{k d (R - f(\frac{\alpha}{k
d}))}$$ where $\alpha \in k d P$. The $\C^*$ action whose weights we are calculating
corresponds to the standard action in the $w$  coordinate and clearly produces the
stated weights.

\end{proof}

The following Corollary immediately implies Proposition
\ref{INTROPROP}.

\begin{cor} \label{LAMBDAs} The eigenvalues (normalized weights)  $\lambda_{\alpha, k}$
are given by
\begin{equation}  \lambda_{\alpha, k} = k d ( R -
f(\frac{\alpha}{kd}))- \frac{1}{d_k} \sum_{\alpha \in k P} k d( R
- f(\frac{\alpha}{kd})).
\end{equation}
\end{cor}

\section{\label{LDMUKZ} The measures $d\mu_k^z$}

In this section we discuss  the measures $d\mu_k^z$
(\ref{MUKZDEF}). Our first purpose  is to give the precise
statement of Theorem \ref{LDINTRO}, and to recall the relevant
definitions from the theory of large deviations \cite{dH, DZ}. We
then prove Proposition \ref{LLNMUZ} using Bergman kernels and
Berstein polynomials. Without loss of generality, we assume that $d=1$ to simplify the calculation.

A function $I: E \to [0, \infty]$ is called a rate function if  it
is proper and lower semicontinuous.  A sequence  $\mu_{k}$
($k=1,2,\ldots$) of  sequence of probability measures on  a space
$E $ is said to satisfy the {\it large deviation principle with
the rate function} $I$ (and with the speed $k$) if the following
conditions are satisfied:
\begin{itemize}
\item[(1)] The level set $I^{-1}[0,c]$ is compact for every $c \in
\R$. \item[(2)] For each closed set $F$ in $E$, $ \limsup_{k\to
\infty}\frac{1}{k}\log \mu_{k}(F) \leq -\inf_{x \in F}I(x). $
\item[(3)] For each open set $U$ in $E$, $ \liminf_{k \to
\infty}\frac{1}{k}\log \mu_{k}(U) \geq -\inf_{x \in U}I(x). $
\end{itemize}
Heuristically, in the sense of logarithmic asymptotics,  the
measure $\mu_k$ is a kind of integral of $e^{- k I(x)}$ over the
set.

In \cite{DE}, Dupuis-Ellis gave an alternative definition in terms
of Laplace type integrals and in particular gave a definition of
uniform large deviations which is very suitable for our problem.
We will state it only in our setting, where the parameter space is
the compact toric variety $M$. Put: \begin{equation}
\label{F(z,h)} F(z, h) = - \inf_{x \in P} (h(x) + I^z(x)).
\end{equation}  Then $d\mu_k^z$ satisfies the {\it Laplace principle on
$P$ with rate function $I^z$ uniformly on $M$} if, for all compact
subsets $K \subset M$ and all $h \in C_b(P)$ we have:

\begin{itemize}
\item[(1)] For all  $c \in \R$, $\bigcup_{z \in M}
(I^z)^{-1}[0,c]$ is compact for every $c \in \R$. \item[(2)] For
each $h \in C_b(P)$,  $ \limsup_{k\to \infty} \sup_{z \in M}
\left(\frac{1}{k}\log \int_P e^{- k h} d\mu_k^z -  F(z, h) \right)
\leq 0. $ \item[(3)] For each $h \in C_b(P)$,    $ \liminf_{k \to
\infty} \inf_{z \in M} \left( \frac{1}{k}\log \int_P e^{- k h(x)}
d\mu_k^z (x) - F(z, h)\right) \geq 0. $
\end{itemize}
The upper and lower bounds of course imply, for  each $h \in
C_b(P)$, $$ \lim_{k \to \infty} \sup_{z \in M} \left(
\frac{1}{k}\log \int_P e^{- k h(x)} d\mu_k^z (x) - F(z, h)\right)
=  0. $$

The probability measures of concern in this article are  the
measures (\ref{MUKZDEF}), which we often write in the form
\begin{equation} \mu_k^z = \frac{1}{\Pi_{h_0^k}(z,z)}\;\; \sum_{\alpha \in k P}
\pcal_{h^k}(\alpha, z)  \; \delta_{\frac{\alpha}{k}},
\end{equation} where $\pcal_{h^k}(\alpha, z) $ is given in
(\ref{PHK}). It is simple to see that $\mu_k^z$ is a probability
measure since the total mass of the numerator equals
$\Pi_{h^k}(z,z)$.










We note that the formula for $\mu_k^z$ simplifies when $z \in
\dcal$:

\begin{prop} \label{FACE} If $\mu(z) \in F$ where $F$ is a face of $P$, then
$$ \mu_k^z = \frac{1}{\Pi_{h_0^k}(z,z)}\;\; \sum_{\alpha \in k
\bar{F}} \pcal_{h^k}(\alpha, z)  \; \delta_{\frac{\alpha}{k}}.
$$
In the extreme case where  $z$ is a fixed point of the $\T$ action
and $\mu_0(z)$ is a vertex,  the measures $\mu_k^z$ always equal
$\delta_{\mu_0(z)}$.
\end{prop}

\begin{proof} It follows easily from
 \cite{STZ} that  $s_{\alpha}(z) = 0$
for all $\alpha$ such that $\alpha \notin \bar{F}$. In the extreme
case where $\mu(z)$ is a vertex $v$,  the only monomial which does
not vanish at $z$ is the monomial corresponding to the vertex. We
then have
$$ \mu_k^z = \frac{1}{\pcal_{h^k}(v, z) }\;\;  \pcal_{h^k}(v, z)  \;
\delta_{v} = \delta_v.
$$

\end{proof}


\subsection{\label{BP}Bernstein polynomials and Proposition \ref{LLNMUZ}}

In this section, we prove Proposition \ref{LLNMUZ} as an
application of Bernstein polynomials in the sense of \cite{Z2}.
We recall here that the  $k$th Bernstein polynomial approximation
to $f \in C(\bar{P})$ was defined in \cite{Z2} by the formula,
\begin{equation} B_{h^k}(f)(x) := \frac{1}{\Pi_{h^k}(z,z)} \sum_{\alpha \in k P} f(\frac{\alpha}{k}) \pcal_{h^k}(\alpha, z), \;\;
x = \mu(z). \end{equation} The definition extends to
characteristic functions of Borel sets $A \subset \bar{P}$ by
$$\mu_k^z(A) = B_k(\chi_A)(x) = \sum_{\alpha \in k P} \chi_A (\frac{\alpha}{k}) \pcal_{h^k}(\alpha, z),\;\;
x = \mu(z). $$

The proof of Proposition \ref{LLNMUZ} is as follows:
\begin{proof} For any $f \in C(\bar{P})$,
$$\int f d\mu_k^z = \sum_{\alpha \in k P} f(\frac{\alpha}{k}) \pcal_{h^k}(\alpha, z)  $$
and the latter is precisely the Bernstein polynomial $B_k (f)(x)$.
In \cite{Z2} it is proved to tend uniformly to $f(z)$.
\end{proof}

We pause to relate Theorem \ref{LDINTRO} to prior results on
Bernstein polynomials for characteristic functions.  In dimension
one, it is a classical result (due to Herzog-Hill) that at a jump
discontinuity, the Bernstein polynomials tend to the mean value of
the jump.  If $A$ is an open set, they converge uniformly to $1$
on compact subsets of $A$, and converge uniformly to zero at an
exponential rate on the compact subsets of the interior of $A^c$
as $k \to \infty$. The large deviations result determines the
exponential decay rate. Intuitively, $I^z(x)$ defines a kind of
distance from $x$ to $z$ using the $(\C^*)^m$ action,  and the
limit $\inf_{x \in A} I^z$ defines  a kind of Agmon   distance
from $z$ to $A$. If $A \subset P^o$ and $z \in
\partial P$ then $B_k(\chi_A)(x) = 0$, and the distance is
infinite. In the case where $P = \Sigma_m$, the unit simplex in
dimension $m$, and with  $h$  the Fubini-Study metric on $\CP^m$
and $A \subset P$ is a convex sublattice polytope, such Bernstein
polynomials were studied under the name of conditional \szego
kernels in \cite{SZ}, since
$$B_k(\chi_A)(x) = \Pi_{h^k| k A}(z,z) =  \sum_{\alpha \in k A}  \pcal_{h^k}(\alpha, z),\;\;
x = \mu(z) $$
is  the diagonal of the  kernel of the orthogonal projection
onto the subspace spanned by monomials $z^{\alpha} $ with $\alpha \in k A$.
 The exponential decay rate was determined
there when $z \in M^o$. In subsequent (as yet unpublished) work,
Shiffman-Zelditch have extended the results of \cite{SZ} to
non-convex subsets as well.

\subsection{Outline of the proof of Theorem \ref{LDINTRO}}

The proof of Theorem \ref{LDINTRO} is based  on the proof of the
G\"artner-Ellis theorem \cite{DZ,dH} and on the Laplace large
deviations principle of \cite{DE}.  The key idea of the
G\"artner-Ellis theorem   is that the rate function should be the
Legendre transform of the scaling limit of the logarithmic moment
generating function. But a key component, the uniformity in $z$,
is not a standard feature of the proof and indeed the lower bound
is non-uniform (and the upper bound is).  However, we only need
uniformity of the Laplace large deviations principle in the sense
of \cite{DE}, and this LDP is uniform. Note that the  large
deviations principle for each fixed $z$ is equivalent to   the
Laplace principle for each $z$, but that the uniformity of the
limits is different in the two principles. In outline the proof is
as follows:

\begin{enumerate}

\item In \S \ref{LOGMGFl}, we introduce the logarithmic moment
generating function $\Lambda^z_k$, and in Proposition \ref{LAMBDA}
we determine its scaling limit $\Lambda^z$.

\item We then introduce its Legendre transform $I^z =
\Lambda^{z*}$, and in Proposition \ref{LAMBDAZ*} we calculate
$I^z$.

\item In \S \ref{PROOFLD} we prove the  large deviations principle
of Theorem \ref{LD} for fixed $z$.

\item IN \S \ref{ULP}, we use a  special analysis of the weights
underlying the measure $d\mu_k^z$ given in Lemma \ref{IZLOGPCAL}
to prove the uniform Laplace large deviations principle.

\end{enumerate}

Let us also   give a  heuristic proof for $z \in M^o$ before the
formal one. Writing $z = e^{\rho/2 + i \theta}$, we then have
$$\pcal_{h^k}(z, \alpha) \sim k^{-m/2} e^{k \left(\langle
\frac{\alpha}{k}, \rho \rangle - u_{\phi}(x) - \phi(z) \right)},
$$ where $u_{\phi}$ is the symplectic potential corresponding to the \kahler potential
$\phi$.  Hence, for any set $A \subset \bar{P}$,
\begin{equation} \label{PHIKTASY} \frac{1}{k} \log \mu^z_k(A) =
\frac{1}{k} \log \sum_{\alpha \in  k A  \cap  \Z^m} e^{- k
(u_{\phi}(\frac{\alpha}{k })  + \langle \rho, \alpha \rangle)} -
\phi(z) + O(\frac{\log k}{k}).
\end{equation}
Visibly,  the rate function $I^z$ is the function on $\bar{P}$
defined by $ u_{\phi}(x) - \langle \rho, x\rangle + \phi(z).$ It
  is a convex proper function of $x \in P$.

If $z \in \mu_0^{-1}(F)$ for some face $F$, then by  Proposition
\ref{FACE}  the previous argument goes through as long as we
restrict all the calculations to $\alpha \in F$. This gives
Theorem \ref{LDINTRO} formally in all cases.

\section{\label{LOGMGFl} Large deviations principle for $\mu_k^z$ for each $z$}

In this section we prove the pointwise large deviations principle
stated in Theorems \ref{LDINTRO} and \ref{LD}. We begin by
discussing the logarithmic moment generating function and its
scaling limit. These are key ingredients in the G\"artner-Ellis
theorem, which we use to conclude the proof. Uniformity in $z$
will be considered in the following section.

In this section we consider the   moment generating function (with
$t \in \R^m$)
\begin{equation} \label{MMG} \begin{array}{lll} M_{\mu_k^z}(t) : & = &  \int_{P} e^{\langle t, x
\rangle} d\mu_k^z(x) \\ && \\
& = & \sum_{\alpha \in k P} e^{\langle \frac{ \alpha}{k}, t
\rangle} \frac{\pcal_{h^k}(\alpha, z)}{\Pi_{h^k}(z,z)}.
\end{array}
 \end{equation}
 Clearly, $M_{\mu_k^z}(t)$ is a convex function of $t$ and is a  Bernstein polynomial in the sense of \cite{Z2} for the
 function $f_t(x) = e^{\langle t, x \rangle}$, and by Proposition
 \ref{LLNMUZ}, $M_{\mu_k^z}(t)  \to e^{\langle t, x \rangle}$
 uniformly in $z$ and $t \in [0, L]$ and $k \to \infty$. However, the relevant limit
 \begin{equation} \label{LOGMGF} \Lambda^z (t): = \limsup_{k \to \infty}
 \Lambda^z_k(t),\;\;\mbox{with}\;\;
 \Lambda_k^z(t) : = \frac{1}{k} \log
M_{\mu_k^z}(k t) \end{equation}
 is the scaling limit of the   logarithmic moment generating
 function.

By a  simple and well-known application of H\"older's inequality (see e.g. \cite{E}, Proposition IVV.1.1),     $\Lambda_k^z(t)$ and $
\Lambda^z(t)$ are convex functions on $\R^m$ for every $z$.

\begin{prop}\label{LAMBDA} We have  $$  \sup_{z \in M} \left| \Lambda_k^z(t) -
\Lambda^z (t) \right| = o(1), \;\; \mbox{uniformly as} \;\; k \to
\infty,
$$
where  $\Lambda^z(t)$  is given as follows:

\begin{itemize}

\item For $z = e^{\rho/2 + i \theta}  \in M_0$, the open orbit,
 $\Lambda^z(t) = \phip(z^{ (t + \rho)/2 + i\theta }) - \phip(\rho) $.
  Here, $e^{t/2} z$ denotes the action of the real subgroup $\R_+^m$ on the open
  orbit,
  and $\phip$ is the $\T$-invariant \kahler potential on the open orbit.

  \item For $\mu_0(z) \in F$, a face,  $\Lambda^z(t) = \phi_F( e^{ (t' + \rho')/2+ i\theta}) - \phi_F( e^{\rho'/2})
    $, where $\phi_F$ is the open orbit invariant \kahler potential for the toric \kahler
    subvariety defined by $F$.

    \item When $z$ is a fixed point, then $\Lambda^z(t) = 0$.

    \end{itemize}

   \end{prop}

\begin{proof} When $z$ lies in the open
orbit, we may write $z = e^{\rho/2 + i \theta}$ and $s_{\alpha}(z)
= z^{\alpha} e$ where $e$ is a $\T$-invariant frame satisfying
$||e||_{h_0}^2(z) = e^{- \phip(z)}$.  Then,
\begin{equation} \label{MMGa} \begin{array}{lll} M_{\mu_k^z}(k t)
& = &  \sum_{\alpha \in k P} \frac{e^{\langle \alpha, \rho
\rangle} e^{\langle t, \alpha \rangle } e^{- k
\phip(z)}}{Q_{h_0^k}(\alpha) \Pi_{h_0^k}(z,z)}
 \\
&&
\\ && = e^{k (\phip(e^{t/2}z) - \phip(z))} \sum_{\alpha \in k P}
\frac{ e^{\langle  t + \rho, \alpha \rangle} e^{- k \phip(e^{t/2}
z)} }{Q_{h_0^k}(\alpha)
\Pi_{h_0^k}(z,z)}  \\
&&
\\ && = e^{k (\phip(e^{t/2}z) - \phip(z)))}
\frac{\Pi_{h_0^k}(e^{t/2} z, e^{t/2 }z) }{ \Pi_{h_0^k}(z,z)}.
\end{array}
 \end{equation}
Here, $e^{t/2} z$ denotes the $\C^*$-action (restricted to
$\R_+^m$).

 It follows that
\begin{equation}\label{LAMBDAKZT}  \begin{array}{lll} \Lambda_k^z(t)  && = \phip(e^{t/2}z) -
\phip(z) + \frac{1}{k} \log \Pi_{h_0^k}(e^{t/2} z, e^{t/2 }z) -
\frac{1}{k} \log \Pi_{h_0^k}(z,z)  \\ && \\ && = \phip(e^{t/2}z) -
\phip(z)) + O(\frac{\log k}{k}), \end{array}
\end{equation}
with remainder uniform in $z$ by (\ref{TYZ}).

The calculation for other $z$ is similar, using Proposition
\ref{FACE} to reduce the $\T$ action to the subtoric variety
corresponding to $F$, and replacing $\phip$ by the open orbit
toric \kahler potential $\phi_F$  on the subtoric variety. At the
vertex, the sum reduces to the vertex and the logarithmic moment
generating function equals zero  by Proposition \ref{FACE}. Since
the remainders all derive from a uniform Bergman-\szego kernel
expansion (\ref{TYZ}), there is a uniform limit as $k \to \infty$
for all $z$.

\end{proof}

\subsection{The Legendre duals $\Lambda_k^{z *}$ and  $\Lambda^{z *}$}

The  Fenchel-Legendre transform of a convex function $F$  is
the convex lower semicontinuous convex function defined
by
$$F^{ *}(x)  = \lcal F (x) = \sup_{t \in \R^m}\{ \langle x,
t \rangle - F(t) \}. $$
 We are concerned with the convex functions $F(t) =
\Lambda_k^z(t)$ and $F(t) = \Lambda^z(t)$.

In the following Proposition, we refer to the {\it relative interior}
  $r i (A)$ of a set $A \subset \bar{P}$. By definition, $ri(P) = P^o$ and for a face $F$,
  $ri (F) = F^o$, the interior of $F$ viewed as a convex subset of
  the affine space of the same dimension which it spans.

\begin{prop} \label{LAMBDAZ*} $\Lambda^{z *}(x) = I^z $ is the convex function
on $\bar{P}$ given by the following:

\begin{enumerate}

\item When    $z = e^{\rho/2 + i \theta}$ lies in the open orbit,
then $\Lambda^{z*}(x) =    u_{0}(x)  + \phip(e^{\rho/2+ i\theta}) - \langle x,
\rho \rangle$ for all  $x \in \bar{P}$ and $\dcal(\Lambda^{z*}) =
\bar{P}$;

\item When    $\mu_0(z)$ lies in a face $F$, and $z = e^{\rho'/2 +
i \theta'}$ with respect to orbit coordinates on $\mu_0^{-1}(F)$,
then $\Lambda^{z*}(x) =    u_{F}(x) + \phi_F(e^{\rho'/2+ i\theta}) - \langle x',
\rho' \rangle$ when $x \in F$ and $\Lambda^{z*}(x) = \infty$ if $x
\notin \bar{F}$. Thus, $\dcal_{\Lambda^{z*}} = \bar{F}$ (cf.
(\ref{DOMAIN}))

\item When $\mu_0(z) $ is a vertex $v$, then $\Lambda^{z^*}(v) =
0$ and $\Lambda^{z*}(x) = \infty$ if $x \not= v$ and
 $\dcal_{\Lambda^{z*}} = \{v\}.$

 \item For each $z \in M$, and for any $x \in ri \dcal(\Lambda^{z*})$,
there exists $t = t_*(x, z) \in \R^m$ so that
 $\nabla_t I^z(t) = x$.

\end{enumerate}

\end{prop}

\begin{proof}

\noindent{\bf (1)}  If $z = e^{\rho/2 + i \theta}$  lies in the
open orbit,
\begin{equation} \label{LAMBDA*} \Lambda^{z*}(x) = \sup_{t \in \R^m} \left( \langle x, t
\rangle - \phip(e^{ t/2}
z) \right) + \phip(z).
\end{equation} We observe that $ \langle x, t \rangle - \phip(e^{ t/2}
z)$ is concave in $t$ and that
$$\nabla_t \phip(e^{ t/2}
z) = \nabla_{\rho} \phip(e^{ t/2}
z) =
\mu_0(e^{t/2} z). $$   The supremum  in (\ref{LAMBDA*}) can only
be achieved at the unique critical point $t = t_*(x,z)$ such that
$$ x = \mu_0(e^{t/2} z) \iff e^{t/2} z = \mu_0^{-1} (x).  $$
 We note that $e^{t/2} z$ lies in the open orbit. If $x \in P^o$,
then there exists a unique $t = t_*(x,z)$ (denoted $\tau_{x}^P(z)$
in \cite{STZ}) such that  $\mu_0^{-1}(x) = e^{t/2} z$, given by
\begin{equation} \label{tstar} t_*(x,z)  = \log \mu_0^{-1}(x) - \log
|z|. \end{equation}
 In this case,  we have
$$\begin{array}{lll} \Lambda^{z*}(x) && =  \langle x,
t_*(x,z) \rangle  -  \phip( e^{t_*(x,z)/2}z)  + \phip(z)
\\ && \\ & = & \langle x,
\log \mu_0^{-1}(x) \rangle  -  \phip(2 \log \mu_0^{-1}(x))   -
\langle x, \rho \rangle + \phip(z)) \\ && \\ & = &   u_0(x)
 + \phip(z) - \langle x, \rho \rangle.
\end{array} $$
This proves (1) when $x \in P^o$.

Now consider the case where $z \in M^o$ and  $x \in \partial P$.
We note that $L^z(x) :=  u_0(x) - \langle x, \rho \rangle +
\phip(z)$ is a continuous function of $x \in  \bar{P}$. We
claim that $\Lambda^{z*}(x) $ is continuous, i.e. it continues to
equal this function when $x \in
\partial P$, where $u_{0}(x) = 0$.
Since the closure of the open orbit is all of $M$, there exists a
one parameter subgroup $\tau \omega$ with $|\omega| = 1, \tau \in
\R$ so that $\lim_{\tau \to \infty} \mu_0(e^{\tau \omega} z) = x$.

 Since $\Lambda^{z*}(x) $ is
lower semicontinuous, we automatically have
$$\Lambda^{z*}(x) \leq \liminf_{\tau \to \infty} \Lambda^{z*} ( \mu_0(e^{\tau
\omega} z)) = L^z(x). $$
To prove the reverse inequality, we use that
$$\Lambda^{z*}(x) \geq \lim_{\tau \to \infty} \left( \langle x, \tau \omega \rangle - (\phip( e^{\tau \omega /2} z)
- \phip(z)) \right), \;\; z = e^{\rho/2 + i \theta}. $$ We now
claim that
$$\left( \langle x, \tau \omega \rangle - (\phip(e^{\tau \omega /2} z)
- \phip(\rho)) \right) - L^z(x) = \langle x, \rho + \tau \omega
\rangle - \phip(e^{\tau \omega /2} z) - u_0(x) \to 0, \;\; \mbox{as}
\; \tau \to \infty. $$ Indeed, $u_0(\mu_0(e^{\tau \omega/2} z)) =
\langle \mu_0 (e^{\tau \omega/2} z), \rho + \tau \omega\rangle -
\phip (e^{\tau \omega /2} z)$, and $u_0$ is continuous, so the claim
reduces to showing that $ \langle x - \mu_0 (e^{\tau \omega/2} z),
\rho + \tau \omega \rangle \to 0.$ Near the boundary, we may
approximate the moment map by that of the linear model to check
that the expression tends to zero exponentially fast. For
instance, in one dimension, with $x = 0$ and $|z| = 1 = - \omega$
the expression becomes $ e^{- \tau} (\rho - \tau).$ It follows
that  $\Lambda^{z*}(x) \geq L^z(x).$ This proves (1) and also (4)
when $x \in \bar{P}$ and $z \in M^o$.

\medskip

\noindent{\bf (2)} Now let us consider the case where $z \in F^o$.
We then consider the toric subvariety equal to the closure of
$\mu_0^{-1}(F)$ in $M$. We pick a base point on this toric
subvariety and consider orbit coordinates $z' = e^{\rho'/2 + i
\theta'}$ for the quotient of $\T\backslash \T_z$ where $\T_z$ is
the isotropy group of $z$. Then,
\begin{equation} \label{LAMBDA*2} \Lambda^{z*}(x) \geq \sup_{t \in \R^m} \left( \langle x,
t \rangle - \phi_F(t' + \rho') \right) - \phi_F(\rho').
\end{equation}
Here, we note that $e^t \cdot z  = e^{t'} \cdot z$ where $e^{t'}$
is a representative of $e^t$ in $\T\backslash \T_z$.

As above, we find that the supremum is only achieved when $x \in
F$ and then the calculation becomes the open orbit calculation for
the sub-toric variety.  By the previous argument, the open orbit
formula extends to the closure of $\mu_0^{-1}(F)$ by continuity
and in addition (4) holds for $x \in F^o$.

On the other hand, if $x \notin F$ then the supremum is not
achieved. Write $x = (x', x'')$ and similarly for $t$. Then
(\ref{LAMBDA*2}) is the sum of $\langle x'', t'' \rangle$ plus
terms depending only on $t'$. If $x'' \not= 0$ we can let $t'' = r
x''$ with $r > 0$ and find that the supremum equals $+ \infty$.

\medskip

\noindent{\bf (3)} In this case, $\Lambda^z = 0 $ and it is
obvious that the supremum is infinite if $x \not= 0$. Here, the
coordinates are chosen so that the vertex occurs at $0$.

\end{proof}

\begin{rem} As a check on signs, we note that $\Lambda^{z*}(x)$ should be non-negative and
convex. Indeed, $\Lambda^{z^*}(x)$ is convex and takes its minimal value of $0$ at  $\mu(z) = x$.
\end{rem}

We end this section with the following  important ingredient in
the uniform estimates:
 \begin{prop} \label{LAMBDAVSLAMBDAK} We have
$$\Lambda^{z *}_k(t) = \Lambda^{z*}(t) + O(\frac{\log k}{k}),\;\; \Lambda^{z *}_{k \delta}(t)
 = \Lambda^{z*}_{\delta} (t) + O(\frac{\log k}{k})$$
where the remainder is uniform in $z,t$ (and $\delta$).
\end{prop}

\begin{proof} By (\ref{LAMBDAKZT}) and by (\ref{TYZ}), we have

\begin{equation} \begin{array}{lll} \Lambda_k^{z*}(x)  && = \sup_{t \in
\R^m}\left( \langle t, x \rangle -  \phip(e^{t/2} z)  -
\frac{1}{k} \log \Pi_{h_0^k}(e^{t/2} z, e^{t/2 }z)  \right) +
\phip(z) + \frac{1}{k} \log \Pi_{h_0^k}(z,z)\\&& \\
&& = \sup_{t \in \R^m}\left( \langle t, x \rangle - \phip(e^{t/2}
z)    \right) + \phip(z) + O(\frac{\log k}{k}) .
\end{array}
\end{equation}

\subsection{\label{PROOFLD}Proof of Theorem \ref{LDINTRO}: Pointwise large deviations }

We now prove that for each $z$, $d\mu_k^z$ satisfies the large
deviations principle with rate function
\begin{equation} \label{RATE} I^z = \Lambda^{z *}. \end{equation}
  We use the
two notations interchangeably in what follows. The proof is an
application of the G\"artner-Ellis theorem as in \cite{DZ,dH,E}.
We postpone a discussion of uniformity to the next section.  For
the sake of completeness we recall the statement of the
G\"artner-Ellis theorem:\medskip

 {\it Let  $\mu_k$ be a sequence of probability measures on $\R^n$. Assume that
 \begin{enumerate}

 \item  the scaled logarithmic   moment generating function
 $\Lambda(t)$ of $\{d\mu_k\}$
  exists and that $0$ lies in its domain   $\dcal(\Lambda)$;

  \item
$\Lambda$ is lower-semicontinuous and differentiable on $\R^n$.
\end{enumerate} Then $\mu_k$ have an LDP with speed $k$ and rate function
$\Lambda^*$.} Here,  the domain of a convex function $F$ is
defined by
\begin{equation} \label{DOMAIN} \dcal_{\Lambda} = \{t \in \R^m: F(t) < \infty\}.
\end{equation}

\medskip

\noindent {\it Proof of the large deviations statement}:  It
suffices to  verify that the hypotheses of the G\"artner-Ellis
theorem are satisfied. All the work has been done in the previous
section.

It follows from Proposition \ref{LAMBDA} that $\Lambda^z(t)$
 satisfies the assumption (1)  of the G\"artner-Ellis theorem:
  the limit (\ref{LOGMGF}) exists for all $t$, and
  the origin
belongs to its {\it domain} $\dcal_{\Lambda^z}$.   Indeed, for all
$z \in M$, $\dcal_{\Lambda^z} = \R^m$.

Further,  $\Lambda^z$ is differentiable everywhere for every $z$
with
\begin{equation} \label{NOTSTEEP} \nabla_t \Lambda^z(t) = \nabla \phip (e^{ t/2}z) = \mu_0(e^{t/2} z). \end{equation}
It follows that $\mu_k^z$ satisfy for each $z$ the LDP with speed
$k$ and rate function $I^z = \Lambda^{z*}$.

\qed

\section{\label{ULP} Uniform Laplace Principle and uniform Varadhan's Lemma}

We now consider uniformity of the large deviations principle and
in particular, the key issue of uniformity of Varadhan's Lemma.
Our goal is to prove the first part of Theorem \ref{GEO}, which we
restate for clarity:
\begin{theo} \label{C0} For any $L > 0$,  $\sup_{(z, t) \in M
\times [0, L]} |\psi_k(t, z) - \psi_t(z)| = 0. $ \end{theo}

In view of  Proposition \ref{INTROPROP}, we need to understand the
convergence of the sequence of relative  \kahler potentials
(\ref{PHIKTZINT}), of which the key issue is the convergence of
\begin{equation}\label{TILDEPHI}  \tilde{\psi}_k(z,t): =  \frac{1}{k} \log  \int_P e^{k t (R - f(x))}
d\mu_k^z(x).
\end{equation}

 We  prove the $C^0$ convergence of $\psi_k \to \psi$ as an
application of Varadhan's Lemma to (\ref{TILDEPHI}) (cf.
\cite{dH}, Theorem III. 13). We recall the statement of the Lemma:
\medskip

\noindent{\bf Varadhan's Lemma} {\it Let $d\mu_k$ be probability
measures on $X$ which satisfy the LDP with rate $k$ and rate
function $I$ on $X$. Let $F$ be a continuous function on $X$ which
is bounded from above. Then
$$\lim_{k \to \infty} \frac{1}{k} \log \int_{X} e^{k F(x)}
d\mu_k(x) = \sup_{x \in X} [F(x) - I(x)]. $$ }

It follows immediately from Varadhan's Lemma and the pointwise
large deviations result of \S \ref{LOGMGFl} that $\psi_k(t, z) \to
\psi(t,z)$ pointwise for each $(t,z)$. However we would like to
prove uniform convergence for $t \in [0, L]$ and $z \in M$.

 In proving $C^0$ convergence, we do not use any special
 properties of $t (R - f)$ beyond the fact that it is a continuous
 function on the closed polytope $P$. Hence, our uniform
 convergence proof automatically implies the uniform Laplace
 principle stated at the beginning of \S \ref{LDMUKZ}.

 \subsection{Weights and rates}

\end{proof}

 The
following Lemma reflects the fact that the weights of our special
measure are already very close to the rate function:

\begin{lem} \label{IZLOGPCAL}  Let  $L_k(z, \frac{\alpha}{k}) = \frac{1}{k} \log
\frac{|s_{\alpha}(z)|^2_{h^k}}{Q_{h^k}(\alpha)} +
I^z(\frac{\alpha}{k})$. Then $L_k(z, \frac{\alpha}{k}) = -
\frac{1}{k} \log Q_{h^k}(\alpha) + u(\frac{\alpha}{k})$ for   $z
\in M^o$ and  satisfies
$$L_k(z, \frac{\alpha}{k}) = O(\frac{1}{k}), \;\; (k \to \infty) $$
uniformly in $z \in M^o$ and $\alpha \in k P$. The same formula
and uniform asymptotics hold when
  $\frac{\alpha}{k} \in F$ and $z \in \mu^{-1}(F)$.

\end{lem}

\begin{proof}

First assume that  $z \in M^o$. Then,  \begin{equation}
\label{ZINT}\begin{array}{lll} \frac{1}{k} \log
\frac{|s_{\alpha}(z)|^2_{h^k}}{Q_{h^k}(\alpha)} &= & \langle
\frac{\alpha}{k}, \rho \rangle - \phip(\rho) - \frac{1}{k} \log
Q_{h^k}(\alpha),\end{array}
\end{equation}
 while
\begin{equation} \label{IZALPHA} I^z(\frac{\alpha}{k}) = - \langle
\frac{\alpha}{k}, \rho \rangle + u_0(\frac{\alpha}{k}) + \phip
(\rho).\end{equation}  Hence,
\begin{equation} \label{IZALPHAa}\frac{1}{k} \log
\frac{|s_{\alpha}(z)|^2_{h^k}}{Q_{h^k}(\alpha)} +
I^z(\frac{\alpha}{k}) = - \frac{1}{k} \log Q_{h^k}(\alpha) + u_0
(\frac{\alpha}{k}). \end{equation}  We observe that the right side
is independent of $z$ and extends continuously from
$\frac{1}{k}$-lattice points $\frac{\alpha}{k}$ to general $x \in
P$,  proving the first statement of the Proposition.

Now suppose that $z \in F$, a face of $\partial P$. Then
$s_{\alpha}(z) = 0$ and  $\log
\frac{|s_{\alpha}(z)|^2_{h^k}}{Q_{h^k}(\alpha)} = - \infty$ unless
$\frac{\alpha}{k} \in \overline{F}$. Also, $I^z(\frac{\alpha}{k})
= + \infty$ when $\frac{\alpha}{k} \notin \overline{F}$. The
formula above gives a meaning to $L_k(z, \frac{\alpha}{k})$ in
this case. When $z \in F$ and $\frac{\alpha}{k} \in \overline{F}$
then in slice-orbit coordinates for $F$,
$$\frac{1}{k} \log \frac{|s_{\alpha}(z)|^2_{h^k}}{Q_{h^k}(\alpha)} =
\langle \frac{\alpha''}{k}, \rho'' \rangle - \phi_F(\rho'') -
\frac{1}{k} \log Q_{h^k}(\alpha), $$ while
$$I^z(\frac{\alpha}{k}) = - \langle \frac{\alpha''}{k}, \rho'' \rangle
+ u_0 (\frac{\alpha}{k}) + \phi_F(\rho''). $$ Thus, the same
formula holds in the case $z \in \dcal$. Indeed, the extension of
$L_k(z, \frac{\alpha}{k})$ by constancy to $M \times P$ is
consistent with its extension using the definition of $I^z$. The
estimate (\ref{QQa}) then completes the proof.

 \end{proof}

\subsection{Uniform large deviations upper bound}

In this section we prove that the large deviations upper bounds
are uniform.

\begin{prop}\label{UNIFUB}  For any compact subset $K \subset \bar{P}$, we have the uniform  upper bound
$$\frac{1}{k} \log \mu_k^z(K) \leq -  \inf_{x \in K}
I^z(x) +   O(\frac{\log k}{k}), $$ where the remainder is uniform
in $z$.

\end{prop}

\begin{proof} By definition,
$$\frac{1}{k} \log \mu_k^z(K) = \frac{1}{k} \log \sum_{\alpha \in k P: \frac{\alpha}{k} \in K}
\frac{\pcal_{h^k}(\alpha, z)}{\Pi_{h^k}(z,z)} . $$
 By Proposition \ref{FACE}, if $z \in \dcal$ then $\mu_k^z$ is
supported on the face of $\mu(z)$. Hence  for any $z$, Lemma
\ref{IZLOGPCAL} applies to all terms in the sum for $\mu_k^z$ and
termwise  we have
$$\pcal_{h^k}(\alpha, z) = e^{- k I^z(\frac{\alpha}{k})} e^{O(1)}, $$
where $O(1)$ is uniform in $z, \alpha$. Since $\log \Pi_{h^k}(z,z) = O(\log k)$ uniformly
in $z$, it  follows that
$$\frac{1}{k} \log \mu_k^z(K) = \frac{1}{k} \log \sum_{\alpha \in k P: \frac{\alpha}{k} \in K}
e^{- k I^z(\frac{\alpha}{k})}  + O(\frac{1}{k}) + O(\frac{\log k}{k}) $$
The Proposition then follows from the facts that
 $e^{- k I^z(\frac{\alpha}{k})}  \leq e^{- k \inf_{x \in K} I^z(x)}$ for every term in the sum,
and that  the number of terms in the sum is $O(k^m)$.
\end{proof}

\subsection{Lower bound}

Unfortunately, the lower bound in the large deviations principle
for open sets is not uniform in $z$. However, for Varadhan's
Lemma, we only need a special case. The following gives an example
of it, although we will need to generalize it later. For $z \in M,
x_0 \in P$, and $\epsilon
> 0$, put
\begin{equation} U( z, x_0, \epsilon) = \{x \in P: I^z(x)  <
I^z(x_0)  + \epsilon\}. \end{equation}

\begin{prop}\label{UNIFLB}  We have
\begin{equation} \label{LBPUNCHLINE}
\frac{1}{k} \log \mu_k^z(U(z, x_0, \epsilon) ) \geq - I^z(x_0) -
  \epsilon +  o(1),
\end{equation} where the remainder is uniform in $z$ and tends to
zero as $k \to \infty$.
\end{prop}

\begin{proof} First, $U(z, x_0, \epsilon) \cap \frac{1}{k} \Z^m
\not= \emptyset$ and indeed,  there exists $\alpha_k(z) \in U(z,
x_0, \epsilon) \cap \frac{1}{k} \Z^m $ satisfying:
\begin{enumerate}

\item $|x_0 - \alpha_k(z)| \leq \frac{\sqrt{m}}{k}$;

\item When $z \in M^o,  \langle \alpha_k(z) - x_0, \log |z|
\rangle \geq 0. $ When $z \in \mu^{-1}(F)$ then $\langle
\alpha_k'(z) - x_0, \log |z'| \rangle \geq 0.$

\end{enumerate}

As in the upper bound,

\begin{equation}
\begin{array}{lll} \frac{1}{k} \log \mu_k^z(U(z, x_0, \epsilon) ) &
= & \frac{1}{k} \log \sum_{\alpha \in k P: I^z(\frac{\alpha}{k}) <
I^z(x_0) + \epsilon} \frac{\pcal_{h^k}(\alpha, z)}{\Pi_{h^k}(z,z)}
\\ && \\
& = & \frac{1}{k} \log \sum_{\alpha \in k P: I^z(\frac{\alpha}{k})
< I^z(x_0) + \epsilon} e^{- k I^z(\frac{\alpha}{k})}  +
O(\frac{1}{k}) + O(\frac{\log k}{k}) \\ && \\
& \geq & \frac{1}{k} \log \sum_{\alpha \in k P:
I^z(\frac{\alpha}{k}) < I^z(x_0) + \epsilon} e^{- k (I^z(x_0) +
\epsilon) } + O(\frac{1}{k}) + O(\frac{\log k}{k}) \\ && \\
& \geq & - I^z(x_0) + \epsilon + O(\frac{1}{k}) + O(\frac{\log
k}{k}).
\end{array} \end{equation}

In the last inequality, we used that
$$I^z(\alpha_k) - I^z(x_0) = - \langle \alpha_x - x_0, \log |z|^2
\rangle + u_0(x) - u_0(\alpha_k), $$ and used the continuity of
$u_0$ to see that $|I^z(\alpha_k) - I^z(x_0)| \leq o(1)$ as $k \to
\infty$.

\end{proof}

\subsection{Proof of Theorem \ref{C0}}

We now prove Theorem \ref{C0}. We employ the notation $F_t$ as in
(\ref{FT}). We derive it from the more general uniform Laplace
large deviations principle:  for each $h \in C_b(P)$,
$$ \lim_{k \to \infty} \sup_{z \in M} \left( \frac{1}{k}\log
\int_P e^{- k h(x)} d\mu_k^z (x) - F(z, h)\right) =  0, $$ where
$F(z, h)$ is defined in (\ref{F(z,h)}). Theorem \ref{C0} is the
special case with $h = - F_t$.

\begin{proof}

We begin with the uniform lower bound,  first proving a
generalization of Proposition \ref{UNIFLB}:

\begin{lem}\label{UNIFLBh}  Let $h \in C(P)$. Then,
\begin{equation} \label{LBPUNCHLINE}
\frac{1}{k} \log \sum_{\alpha \in k P: I^z(\frac{\alpha}{k}) +
h(\frac{\alpha}{k}) < I^z(x_0) + h(x_0) + \epsilon} e^{- k
h(\frac{\alpha}{k}) } \frac{\pcal_{h^k}(\alpha,
z)}{\Pi_{h^k}(z,z)} \geq - I^z(x_0) - h(x_0) + \epsilon + o(1),
\end{equation} where the remainder is uniform in $z$ and $x_0$ and tends to
zero as $k \to \infty$.
\end{lem}

\begin{proof} With no loss of generality we may assume $h \geq 0$, i.e.
$e^{- k h(x)} \leq 1$; if not, we may replace $h$ by $h - \min
h$. Then the left side is

\begin{equation}
\begin{array}{l} = \frac{1}{k}  \log \sum_{\alpha \in k P: I^z(\frac{\alpha}{k}) +
h(\frac{\alpha}{k}) < I^z(x_0) + h(x_0) + \epsilon} e^{- k
h(\frac{\alpha}{k}) } e^{- k I^z(\frac{\alpha}{k})}  +
O(\frac{1}{k}) + O(\frac{\log k}{k}) \\  \\
\geq \frac{1}{k} \log \sum_{\alpha \in k P:  I^z(\frac{\alpha}{k})
+ h(\frac{\alpha}{k}) < I^z(x_0) + h(x_0) + \epsilon} e^{- k
(I^z(x_0) + h(x_0) +
\epsilon) } + O(\frac{1}{k}) + O(\frac{\log k}{k}) \\  \\
\geq - I_z(x_0) - h(x_0) + \epsilon + O(\frac{1}{k}) +
O(\frac{\log k}{k}).
\end{array} \end{equation}

\end{proof}

 Since $x_0, \epsilon$ are arbitrary, it
follows that
 \begin{equation}
\label{LBFT} \frac{1}{k} \log \int_P e^{- k h (x) } d\mu_k^z(x)
\geq \inf_{x \in P} \left(- h(x) - I^z(x) \right) + o(1), \;\;
\end{equation}
giving a uniform lower bound.

For the upper  bound, we use that
\begin{equation}
\begin{array}{lll} \frac{1}{k} \log \sum_{\alpha \in kP \cap \Z^m}
e^{k h(\frac{\alpha}{k})} \frac{\pcal_{h^k}(\alpha,
z)}{\Pi_{h^k}(z,z)} &  = & \frac{1}{k}  \log \sum_{\alpha \in k P}
e^{- k h(\frac{\alpha}{k}) } e^{- k I^z(\frac{\alpha}{k})}  +
O(\frac{1}{k}) + O(\frac{\log k}{k}) \\ &&   \\
& \leq & \max_{\alpha \in k P \cap \Z^m}  - \left(
h(\frac{\alpha}{k}) + I^z
(\frac{\alpha}{k}) \right) \\ && \\
& \leq & \max_{x \in P} - \left(h(x) + I^z(x) \right).
\end{array} \end{equation}

\end{proof}

This completes the proof of uniform Laplace large deviations
principle and also of Theorem \ref{C0}.

\section{\label{MM} Moment map $\mu_t$ and subdifferential of $u + t f$}

We  recall that the moment map $\mu = \mu_{\omega}$ of  a smooth
toric \kahler variety is defined by  $\mu_{\omega}(z) = (H_1,
\dots, H_m) \in P \subset \R^m$, with $d H_j = \omega
(\frac{\partial}{\partial \theta_j}, \cdot)$ where
$\{\frac{\partial}{\partial \theta_j}\}$ is a basis for $Lie(\T)$.
On the open orbit it is given by $\mu_{\omega}(e^{\rho/2}) =
\nabla_{\rho} \phi (\rho)$ where $\phi$ is the open orbit \kahler
potential. On a non-open orbit $\mu_{\omega}^{-1}(F)$ for some
face $F$, it is given by the analogous formula but with $\phi_F$
replacing $\phi$ and the $\rho''$ orbit coordinates of \S
\ref{BACKGROUND} replacing $\rho$. Thus, the moment map defines a
stratified Lagrangian torus fibration $\mu: M \to P$ which is a
fibering $M^o \to P^o$ away from the divisor at infinity and
boundary of $P$. In the real picture, if we divide by  $\T$, the
moment map on the quotient,  $\mu_{\omega}: M/\T \to P$ is a
diffeomorphism of manifolds with corners.

Our first purpose in this section is to prove that the interior
and face-wise  gradient maps are well-defined for the singular
potentials $\psi_t$, and  define a moment map $\mu_t$ for the
singular form $\omega_t$
 (cf. (\ref{MUPHI})).
We then study the regularity and mapping properties  of $\mu_t$.
We would like to generalize the homemorphism  property
$\mu_{\omega}: M/\T \to P$ to the singular \kahler forms
$\omega_t$. To do so, we observe that the diffeomorphism from
$M^o/\T \to P^o$ is the standard inverse relation between the
gradient map $\nabla \phip$ defined by the \kahler potential and
the inverse gradient map $\nabla u_0$ defined by its Legendre
transform, the symplectic potential. For the singular metrics
$\omega_t$, the moment maps $\mu_t$ is no longer a homeomorphism
of this kind. Indeed, the gradient  $\nabla (u_0 + t f)$ of its
symplectic potential  is not differentiable on its
codimension-one corner set
$$\ccal = \{x \in \bar{P}, \exists i \not=j : \lambda_i(x) =
\lambda_j(x) \}. $$ The derivative is discontinuous and its image
disconnects the complementary regions. We write $P \backslash
\ccal = \bigcup_{j = 1}^R P_j$, where $u_0 + t f = u_0 + t
(\langle \lambda_j, x \rangle + \eta_j ) $ and refer to $P_j$ as
the jth smooth chamber.

However, the inverse relation between $\mu_t$ and $\partial (u_0 +
t f)$ can be re-instated as a homeomorphism if we replace the
gradient map  $\nabla (u_0 + t f)$ by the set-valued
subdifferential $\partial (u_0 + t f)$ and $P$ by its  graph
$\gcal
\partial (u_0 + t f)$. This explains why $\psi_t$ can be $C^1$
although $u_0 + t f$ is not.

At the boundary $\partial P$, $\nabla u$ has logarithmic
singularities and hence the graph of $\nabla u$, hence of  $\nabla
(u_0 + t f)$, is not well-defined. This is because $\nabla u$ in
the smooth case must invert $\nabla \phi$ and send points of
$\partial P$ to $\dcal$. The polytope $P$ already compactifies
this picture, and this suggests that we replace the graph of the
subdifferential $ \partial (u_0 + t f)$ by the graph $\gcal
\partial (t f)$ of the subdifferential of the relative symplectic
potential, i.e. $\partial (t f)$ over $P$. It is of course
homeomorphic to the graph of $\nabla (u_0 + t f)$ away from
$\partial P$ and compactifies the graph on $\partial P$.

We therefore prove:

\begin{theo} \label{LIP} $\mu_t$ is Lipschitz continuous on $M$ for each $t \geq
0$. Moreover, $\mu_t$ has a natural lift  $\tilde{\mu}_t : M \to t
\gcal \partial f $ which is a homeomorphism (see Definition
\ref{LIFTEDI}).

\end{theo}

\subsection{Regularity of $\psi_t$ and definition of $\mu_t$ on orbits}

Using standard results of convex analysis, we can immediately
obtain the regularity of $\psi_t$ on the open orbit and along any
other  orbit.

Let us first recall the relation between the relative \kahler
potential $\psi_t$ and the open orbit absolute \kahler potential

 \begin{prop} \label{U+TFC1} For $t \geq 0$,
 \begin{enumerate}

 \item $\lcal (u_0 + t f) \in
 C^1(\R^m)$. Hence $\psi_t|_{M^o} \in C^1(M^o)$.

 \item For any face $F$ of $P$, $\lcal_F (u + t f) \in
 C^1(\R^{m-k})$. Hence $\psi_t|_{M_F} \in C^1(M_F)$.

 \end{enumerate}
 \end{prop}

\begin{proof} The first statement follows immediately from Theorem 26.3 of
 \cite{R}:  {\it A closed proper
convex function is essentially strictly convex if and only if its
Legendre conjugate is essentially smooth}. It suffices to recall
the definitions and to verify that $ u + tf$ is essentially
strictly convex on $P^o$.

A proper convex function $g$ is called {\it essentially strictly
convex} (see \cite{R}, page 253)  if $g$ is strictly convex on
every convex subset of $\mbox{dom} (\partial g) = \{x: \partial g
(x) \not= \emptyset\}.$  Here, $\partial g(x)$ is the
subdifferential of $g$ at $x$.  This property is satisfied by $g =
u + t f$ since $u$  is strictly convex on $\bar{P}$ and since $f$
is convex so that $u + tf $ is strictly convex for $t> 0$. It
follows that  $\lcal(u + t f)$ is essentially smooth.

We  recall that  $g$ is {\it essentially smooth} (\cite{R}, page
251) if its domain $\mbox{dom} (g)$ has non-empty interior $C$, if
$g$ is differentiable on all of $C$ and if $|\nabla g(x_j)| \to
\infty$ when $x_j \to \partial C$ (the last condition is vacuous
if $C = \R^m$.) Thus,  $\lcal(u + t f)$ is differentiable on
$\R^m$. To complete the proof, it suffices to recall that an
everywhere differentiable convex function is automatically $C^1$
(cf. \cite{R}, Corollary 25.5.1).

The same proof shows that the restrictions of $\psi_t$  to the
sub-toric varieties $M_F$ of $\psi_t$ are also $C^1$ along the
sub-toric varieties, proving the second statement.

\end{proof}

\begin{cor} The gradient maps

\begin{itemize}

\item $\nabla_{\rho} \phi_t (\rho): \R^m \to P^o$,

\item $\nabla_{\rho''} \phi_t|_{M_F} (\rho''): \R^{m - k} \to F$

\end{itemize}

are well defined and continuous.
\end{cor}

The  Corollary serves to define the moment map $\mu_t$:
\begin{equation} \label{MUTDEF} \mu_t(z) = \left\{ \begin{array}{ll} \nabla_{\rho}
(\psi_t+ \phip)
(\rho), & z = e^{\rho/2 + i \theta} \in M_0; \\ & \\
\nabla_{\rho''} (\psi_t + \phip) (\rho''), & z = e^{\rho''/2 + i
\theta''} \in M_F. \end{array} \right. \end{equation}

\subsection{\label{MMSUBDIF}Moment map $\mu_t$ and
subdifferential map}

So far, we have shown that $\mu_t$ as defined by (\ref{MUTDEF}) is
continuous in $(t, z)$ on the interior and along each boundary
face. To complete the study of $\mu_t$ we need to prove the
homeomorphism properties and to  analyze continuity across
$\dcal$. We will need to  recall some further definitions in
convex analysis on $\R^m$, following \cite{R,HUL}. First, a convex
function $g$ is called {\it proper} if $g(x) < \infty$ for at
least one $x$ and $g(x) > - \infty$ for all $x$ (cf. \cite{R},
page 24).  When $g$ is a proper convex function, it is {\it
closed} if it is lower semi-continuous (\cite{R}, page 54). These
conditions are trivially satisfied for $g = u + t f$. A
subgradient of $g$ at $x \in \R^m$ is a vector $x^* \in \R^m$ such
that
$$g(y) - g(x) \geq \langle x^*, y - x \rangle, \;\; \forall y, $$
i.e. if $g(x) + \langle x^*, y - x \rangle$ is a supporting
hyperplane to the epi-graph  $\mbox{epi} (g)$ at $(x, g(x)).$ The
subdifferential of $g$ at point $x_0$ is the set of subgradients
at $x$, i.e.
$$\partial g(x_0) = \{\rho: g(y) - g(x) \geq \langle \rho, y - x \rangle, \;\; \forall y\}, $$
 A convex function $g$ is
differentiable at $x$  if and only if $\partial g(x)$  is a single
vector (hence $\nabla g(x)$).

The graph of the subdifferential of $g$  is the set
\begin{equation} \label{GCAL} \gcal (\partial g) = \{(x, \rho): \rho \in \partial
g(x)\} \subset T^* \R^m. \end{equation} In the smooth case, it is
a Lagrangian submanifold. In the cornered case it is a Lagrangian
submanifold with corners. An illustration in the one-dimensional
case may be found \cite{HUL}I (p. 23). As the illustration shows,
the derivative has a jump at each corner point $x$, and the graph
fills in the jump at $x$ with
 a vertical   interval $[D_- f(x), D_+ f(x)]$ where $D_{\mp} f(x) $
 are the left/right derivatives.

\subsection{Legendre transforms,  gradient maps and Lagrangian graphs}

We recall that the Legendre transform of a convex function $g$ on
$\R^m$  is defined by $\lcal g (x) = g^*(x) = \sup_{\rho \in \R^m}
(\langle x, \rho \rangle - g(x) ). $

 When the gradient map $\nabla
g$ is invertible, its inverse is the gradient map $\nabla g^*$.
Another way to put this is that the graph $gr \nabla = (x, \nabla
g(x))$ of $\nabla g$ is a Lagrangian submanifold  $\Lambda \subset
T^*\R^m = \R^{2m}_{x, \xi}$ which projects without singularities
to the base $\R^m_x$. One says that $g$ parameterize $\Lambda$.
Open sets of $\Lambda$ which project without singularities to the
fiber $\R^m_{\xi}$ can be parameterized as graphs $(\nabla
g^*(\xi), \xi)$ of $\nabla g^*: \R^m_{\xi} \to T^*\R^m$. We put
$\iota (x, \xi) = (\xi, x)$ so that $\Lambda$ is (locally) given
as $\iota \gcal \nabla g^*$. One then says that $g^*$
parameterizes $\Lambda$ (locally). Applying $\iota$  is equivalent
to the statement that $\nabla g, \nabla g^*$ are inverse maps.

When $g$ is not differentiable, one can replace the graphs of
$\nabla g$ and $\nabla g^*$ by the sub-differentials $\partial g,
\partial g^*$. The graph of $\partial g$ is a piecewise smooth
Lagrangian manifold $\Lambda$  with corners over the non-smooth
points of $g$. $\Lambda$ is also the graph of $\iota \partial g^*$
over the subset of $\R^m_{\xi}$  where it projects. As a simple
example consider $g(x) = |x|$:  The  graph of $\partial g$
consists of the graph of $\xi = -1$ on $\R_-$ together with the
vertical segment $x = 0, \xi \in (-1, 1)$ together with the graph
of $\xi = 1$ on $\R_+$. We see that the graph projects to the
interval $[-1, 1]  \subset \R_{\xi}$. Hence the domain of $\lcal
|x|$ is $[-1,1]$.  Further, $\lcal |x| = 0$ on $[-1, 1]$ and $\pm
\infty$ for $\xi
> 1,$ resp. $\xi < -1$. Hence, $(\partial \lcal |x|) (\xi) = 0$
for $\xi \in (-1, 1)$ and it is the half-line $x > 1$ at $\xi =
1$, the half line $x < -1$ at $\xi = -1$ and it is undefined
elsewhere. We see that $\iota \gcal \partial \lcal |x| = \gcal
\partial |x|$ and hence in a generalized sense the set-valued maps
$\partial |x|$ and $\partial \lcal |x|$ are inverses.

\subsection{The moment map and subdifferential of $u + t f$ on the interior}

We now return to  the moment map $\mu_t$, where
 the convex function of interest $\psi_t$ (cf. (\ref{PHIT})).

When $t = 0$, the moment map $\mu_0 =   \nabla \phip(\rho) $ is a
homeomorphism from $M^o_{\R} \to P^o$ which is inverted  by $ \exp
\frac{1}{2} d u_{0}$ (cf. (\ref{GRADULOGMU})).  In the non-smooth
case, the
 gradient map is undefined at the corner set. We then
  replace the gradient at the corner by  the subdifferential $\partial
(u_0 + t f)$
 \cite{R,HUL}. The following Proposition identifies this set.
 Given $x$, we denote by $J_x = \{j: \lambda_j(x) =
 \max\{\lambda_1(x), \dots, \lambda_r(x) \}$.

 \begin{prop} \label{CH} Let $f = \max\{\lambda_1, \dots, \lambda_r \}$ be a piecewise affine function on
 $P$. Then for $x \in P^o$,  $$\partial (u_0 + t f)(x) = \nabla u(x)
 + t CH \{\nu_j: j\in J_x\}. $$
 \end{prop}

\begin{proof} Clearly, $\partial (u_0 + t f)(x) = \nabla u_0(x) + t
\partial f(x)$, so it suffices to show that $\partial f(x) = CH \{\lambda_j: j\in J_x\}. $
But it is well known that the subdifferential of the maximum $g =
\max_j g_j$ of $r$  convex functions is the convex hull of the
gradients of those $g_j$ such that $g(x) = g_j(x)$ (see e.g.
\cite{IT} for much more general results).
\end{proof}

\begin{prop} The graph of the subdifferential,  $\gcal \partial (u_0 + t f)(x) \subset T^* P^o$,  is a piecewise smooth
Lagrangian submanifold with corners. The graph of the
subdifferential of  the relative symplectic, $\gcal
\partial (t f)(x) \subset T^* P$, is a piecewise linear
Lagrangian submanifold with corners over $P$.  \end{prop}

\begin{proof}  The term $\nabla u_0$ is clearly
irrelevant to the first statement and hence the second statement
implies the first.

The statement is local and can be checked in the model examples
where the affine functions are the coordinates $x_1, \dots, x_p$
of $ \R^{m}$. The corner set consists of the large diagonal
hyperplanes $x_{1} = \cdots = x_{p} = 0$ for some $p \leq m$
and distinct indices among $1, \dots, m$. On the complement of the
large diagonal hyperplanes, the subdifferential is a constant
vector $e_j$ for some $j$. Over any point $x$  of the diagonal
hyperplanes $x_i = x_j$,  the subdifferential contains the
one-dimensional convex hull $CH(e_i, e_j)$. This convex hull over
the hyperplane is an $m$ dimensional linear manifold bridging  the
two constant graphs and the union is a piecewise linear
Lagrangian submanifold $\Lambda_1$. Over higher $s$-codimension
intersections,  the subdifferential contains $s$ such
one-dimensional convex sets and  equals their convex hull. It
follows that the  full subdifferential is piecewise linear of
dimension $m$. The canonical symplectic form of $T^*M$ vanishes on
all smooth faces and on all vectors tangent to the corners.
\end{proof}

 Although $\partial (u + t f)$ is multi-valued,
$\partial \phi_t (\rho)$ is a singleton for all $\rho$:
\begin{prop}\label{ONEONE}  For each $\rho \in \R^m$, there exists precisely one $x \in P^o$
so that  $(x, \rho) \in \gcal (\partial(u_0 + t f))$. \end{prop}

\begin{proof} This is actually equivalent to Proposition
\ref{U+TFC1} and is the key step in the proof we quoted from
Theorem 26.3 of \cite{R}: a convex function $f$ is differentiable
at $x$ if and only if $\partial f$ consists of just one vector
(i.e. $\nabla f(x)$.)

Let us verify that   $\partial (u_0 + tf) (x_1) \cap
\partial (u_0 + tf) (x_2) = \emptyset$  when $x_1 \not= x_2$.
 We argue by contradiction: suppose
that $x^* \in \partial (u_0 + tf) (x_1) \cap \partial (u_0 + tf)
(x_2)$. Then the graph of $\langle x^*, z \rangle - (u_0 +
tf)^*(x^*)$ is a non-vertical supporting hyperplane to $\mbox{epi}
(u_0 + tf)$ containing $(x_1, (u_0 + tf)(x_1))$ and $(x_2, (u_0 +
tf)(x_2))$. But then the line segment joining these points lies in
the hyperplane, so $(u_0 + tf)$ cannot be strictly convex along
the line segment joining $x_1$ and $x_2$.

\end{proof}

\begin{cor}\label{PHITNOTC1} If $\psi_t$ is associated to a non-trivial
test configuration, then $\psi_t \notin C^2(M)$. \end{cor}

\begin{proof} Indeed, $\psi_t |_{M^o} \notin C^2(M^o)$ since
$\nabla^2 \psi_t$ has a kernel at each point  in the image of the
subdifferential of the corner set but it is strictly positive
definite in the smooth regions.

\end{proof}

\begin{defin} \label{LIFTEDI} The  lifted interior  moment map is the
composite map
\begin{equation} \label{LMM} \begin{array}{lllll}\tilde{\mu}_t(\rho) = (\mu_t(\rho), \rho) :
 M^o &\to & \gcal (\partial(u_0 + tf)) \subset T^* P^o \\ &&\\
&& \downarrow  \\ &&\\  && \gcal (\partial( tf)) \subset T^* P^o .
\end{array} \end{equation}
The downwards arrow is the map $(\mu_t(\rho), \rho) \to
(\mu_t(\rho), \rho - \nabla u_0 (\mu_t(\rho))$.  We define
analogous maps along other orbits where  we replace $\phip$ by
$\phi_F$ and $u_0$ by $u_F$ (the $F$-Legendre transform of
$\phi_F$.)

\end{defin}

\begin{cor} \label{LIFTEDMM} The `lifted moment map' $\tilde{\mu}_t(\rho):  M^o_{\R} \to  \gcal (\partial(u +
tf))|_{P_o} $ is a homeomorphism. The same is true for
restrictions to orbits in $\dcal$ and the corresponding boundary
faces.
\end{cor}

\begin{proof} The graph of $\mu_t = \nabla
\psi_t$ is  homeomorphic to the graph of $\partial (u_0 + t f)$
under $\iota$, since they are Legendre duals: i.e.  $ \iota \gcal
\partial \psi_t= \gcal \partial (u_0 + tf)$. The `projection' to $\gcal(t f)$
obtained by composing with the subtraction map $\rho - \nabla u_0
(\mu_t(\rho))$ is clearly a homeomorphism, hence so is the
composition.
\end{proof}

\subsection{Moment map near the divisor at infinity  $\dcal \cap M_{\R}$}

So far, we have proved that the lifted moment maps are
homeomorphisms from orbits to graphs of the subdifferential of $t
f$ on faces of $P$. Since  such orbits (resp. faces) form a
partition of $M$ (resp. $P$), the remaining steps in the proof of
Theorem \ref{LIP} are to prove  that $\mu_t$ is continuous on
  all of $M$ and has a continuous inverse  $\mu_t^{-1}$ from $P$  to the closure
  of the real open orbit.

We study the behavior of $\mu_t$ near $\dcal$ and the  boundary
behavior of its inverse by exponentiating the subdifferential. To
understand this from an invariant viewpoint, we recall that
 $d u(x) \in T^* (Lie(\T)^*) \simeq Lie(\R_+^m)$, so that we may
 regard $(x, du_0(x)) \in (Lie(\T)^*) \oplus  Lie(\R_+^m)$. In the
 same way, we define the graph of the subdifferential by
 \begin{equation}\label{GRAPHUT}  \gcal (\partial (u_0 + t f): =
  \{(x, \rho)\in Lie(\T)^* \oplus  Lie(\R_+^m): \rho \in \partial(u_0 + t
 f)(x) \}. \end{equation}
We  also view $(\rho, \mu_t (\rho))$ as an element of $
Lie(\R_+^m) \oplus Lie(\T)^*$ and define the graph of $\mu_t$ over
$M^o$  by
\begin{equation} \label{GRAPHDPHIT} \gcal(\mu_t) = \{(\rho, x) \in
Lie(\R_+^m) \times Lie(\T)^*: x \in \partial \psi_t(\rho)\}.
\end{equation} As discussed in \S \ref{BACKGROUNDDP}, we can cover
$M$ with
 affine charts $U_v$ which straighten out the corner of $M_{\R}$
 at each vertex so that it becomes the standard orthant. Hence, we
 will only discuss the fixed point corresponding to the vertex $v
 = 0$ of $P$ and the chart $U_0 \simeq \C^m$  in which the \kahler potential has the
 form (\ref{CANKP}). Recall also  that
$u_0$ is attached to the chart $U_0$, and that there are analogous
potentials for the other charts $U_v$.

 We now define the
 `exponentiated' subdifferential by
 \begin{equation}\label{GRAPHUT}  \ecal (\partial (u_0 + t f)): =
  \{(x, e^{\rho/2})\in P \times Lie(\R_+^m): \rho \in \partial(u_0 + t
 f)(x) \}.  \end{equation}
The basic point is that although  $|d u_0(x)| = \infty $ for $x
\in
\partial P$, $\exp \frac{1}{2} d u_0(x)$ is well-defined as an
element of $\R_+^m$ for $x$ in the corner facets incident at $v =
0$, and indeed $\exp \frac{1}{2} d
 u_0(x)$ has a continuous (in fact, smooth)  extension to the corner
 facets incident at $v = 0$.  This is easily checked by
writing $u_0 = u_P + g$ where $u_P$ is the canonical symplectic
potential and $g \in C^{\infty}(P)$. As an example, we note that
in the case of $\CP^1$, $u' = \log \frac{x}{1 - x} + g'$ blows up
at $x = 0$ while its exponential  $\exp \frac{1}{2} u' =
\frac{x}{1 - x} e^{g'}$ is well defined and takes the value $0$
when $x = 0$ for any $g'$. Essentially the same calculation
verifies that the exponentiated subdifferential has a  smooth
extension to the boundary  for the canonical symplectic potential
plus any piecewise linear convex  function on the polytope of  any
toric variety.

\begin{prop} \label{ECALF}  $\ecal (\partial (u_0 + t f)) \subset P \times \R_+^m$
is a $C^0$ submanifold of $T^* P$ which is homeomorphic to $\gcal
(t f)$.  Its boundary consists of the union over open faces $F$ of
\begin{equation}\label{GRAPHUTF}  \ecal (\partial (u_F + t f)): =
  \{(x'', e^{\rho''/2})\in F \times Lie(\R_+^{m-k}): \rho''  \in \partial(u_F + t
 f)(x) \} \end{equation}
 \end{prop}

 \begin{proof}
 To prove this,  we split the coordinates into
 $(x', x'') \in \R^k \times \R^{m - k}$
  so that $x' = 0$ defines a given face $F$, and
 split the sub-differential vectors in  $\partial  (u + t f)$
 into their $x', x''$ components to obtain component
 subdifferentials
 $\partial', \partial''$. To obtain the limit along $F^o$, we let  $x' \to 0$
 while keeping $x''$ bounded away from zero. Then as $x' \to 0$ and for   $\rho' \in \partial' (u +
 t f)$,  $e^{\frac{1}{2} \rho'} \to 0 \in \R^{k}$. Indeed,
 $e^{\frac{1}{2} d' u(x)} \to 0$ and  the addition of
 $t \partial' f$ only adds a bounded amount to the exponent.

On the other hand, the `slices' $u_0(x', x'') + t f(x', x'')$ for
fixed $x'$, viewed as functions of $x''$, have a subdifferential
$\partial''(u_0(x', x'') + t f(x', x''))$ in $x''$. As is easily
seen from the canonical (or Fubini-Study) symplectic potential,
the subdifferential is bounded as $x' \to 0$ as long as $x''$
stays in the interior of $F$. Further, $\partial''(u_0(x', x'') +
t f(x', x''))$ is a continuously varying $C^0$ submanifold (or
manifold with corners) as $x'$ varies. The same is true if we
exponentiate the subdifferential as in (\ref{GRAPHUTF}). Hence as
$x' \to 0$ the exponentiation of $\partial''(u_0(x', x'') + t
f(x', x''))$ tends $C^0$ to (\ref{GRAPHUTF}). Combining with the
fact that the exponentiated $\partial'$ subdifferential  vanishes
as $x' \to 0$ we conclude that $\ecal (\partial (u + t f))$ in the
interior extends continuously to the corner, where it coincides
with (\ref{GRAPHUTF}). \end{proof}

We now define a lifted moment map in the chart $U_v$; as above, we
may assume $v = 0$.

\begin{defin} \label{LIFTEDUV} The exponentiated  lifted   moment map in $U_0$ is the map
which to $z = e^{\rho/2} \in M^o$ assigns
\begin{equation} \label{LMM} \begin{array}{lllll}\tilde{\mu}_t(z) = (\mu_t(z), e^{\rho/2}) :
 U_0 &\to & \ecal (\partial(u + tf)).
\end{array} \end{equation}
For $z \in \mu^{-1}(F)$ of the form $z = e^{\rho''/2}$ it assigns
\begin{equation} \label{LMM} \begin{array}{lllll}\tilde{\mu}_t(z) = (\mu_t(z), e^{\rho''/2}) :
 U_0 &\to & \ecal (\partial(u + tf)).
\end{array} \end{equation}

\end{defin}

\begin{cor} \label{LIFTEDMMa} The `exponentiated lifted moment map' $\tilde{\mu}_t(\rho) : U_{0 \R} \to  \ecal (\partial(u +
tf))|_{P_o} $ is a homeomorphism. Hence the lifted moment map
$\tilde{\mu_t}$ is a homeomorphism $M \to \gcal(t f)$.
\end{cor}

\begin{proof}

 By construction, the graph of $\mu_t$ on $U_v$ is inverse to the
exponentiated subdifferential on its image.

 \end{proof}

Thus we have proved that $\mu_t$ is continuous and we have
determined its homeomorphism property. The above results also
imply the regularity statement of Theorem \ref{LIP}

\begin{lem} \label{MUTLIP} $\mu_t$ is Lipschitz continuous on all of $M$.

\end{lem}

\begin{proof} This  follows from Corollaries \ref{LIFTEDMM} and
\ref{LIFTEDMMa}:   $\mu_t$ must be Lipschitz
  because its  graph is the $\iota$-image
of the graph of $\partial (u + t f)$ on the interior (and its
exponentiation near the boundary). These graphs are manifestly
given by piecewise smooth manifolds with corners, hence so is the
graph of $\mu_t$.
\end{proof}

We note that the Lipschitz property of $\mu_t$ also follows from a
standard result
 of convex
analysis  on the open orbit and along the divisor faces $M_F$:
Since $u_0 + t f$ is a convex continuous function on $P$, its
Legendre transform, $ \lcal (u_0 + t f) (\rho) $ is a convex lower
semi-continuous function on $\R^m = Lie(\T)$. The same description
is true along the boundary faces of $P$. If $z \in \mu^{-1}(F)$,
$I^z = \infty$ unless $x \in F$, so we may restrict the  supremum
in the Legendre transform to $F$ and it becomes the Legendre
transform $\lcal_F$ along the vector space spanned by $F$.  It
follows that $\lcal_F (u_0 + t f)$ is a convex lower
semi-continuous function on $Lie(\T/\T_z)$, where $\T_z$ is the
stabilizer of $z$.

\begin{lem} The $\rho$-gradient map  $\nabla \psi_t (\rho)$ is  Lipschitz continous  from $\R^m \to P$.
Similarly, $\nabla \psi_{ t}|_{M_F} (\rho'')$ is Lipschitz.
\end{lem}

\begin{proof} The proof is an application of the following fact
(\cite{HUL}, Theorem 4.2.1):
\medskip

{\it If $g: \R^m \to \R$ is strongly convex with modulus $c > 0$,
i.e.
$$g(\alpha x_1 + (1 - \alpha) x_2) \leq \alpha g(x_1) + (1 -
\alpha) g(x_2) - \frac{c}{2} \alpha (1 - \alpha) ||x_1 - x_2||^2$$
then $\nabla g^*$ is Lipschitz with constant $\frac{1}{c}$, i.e.
$$||\nabla g^*(s_1) - \nabla g^*(s_2) || \leq \frac{1}{c} ||s_1 -
s_2||. $$}

On the open orbit, $\psi_t = (u + t f)^*$. We claim that $u + t f$
is strongly convex. Indeed, as in \cite{A,SoZ1}, the Hessian $G =
\nabla^2_x u_{0}$  of the symplectic potential has simple poles on
$\partial P$ and is uniformly bounded below,  $G \geq c I$ for
some $c
> 0$ on $P$.  Then,
$$u(\alpha x_1 + (1 - \alpha) x_2) \leq \alpha u (x_1) + (1 -
\alpha) u(x_2) - \frac{c}{2} \alpha (1 - \alpha) ||x_1 -
x_2||^2.$$ Since $t f$ is convex, it follows that  $u + t f$ is
strongly convex with modulus $c$. It follows that $d \psi_t$ is
Lipschitz on $M^o$ and hence that $\mu_t$ is.

Using the boundary symplectic potentials,  the same proof shows
that $\mu_t |_{M_F} : \mu_0^{-1}(F) \to F$ is Lipshitz continuous
for any face $F$.

\end{proof}

\subsection{\label{EXPLICIT}Explicit formula for $\mu_t$ and $\psi_t$}

It is useful to give   explicit formulae the moment map $\mu_t$
and for $\psi_t$. First, we give the formula on the inverse image
of the smooth domains $P_j$ and then we give the formula on the
inverse image of the corner set.

\begin{prop} \label{MUT} For each $t \geq 0$, the moment map $\mu_t$
defines  a diffeomorphism   $\mu_{t, j}: e^{ t \nu_j}
\mu_0^{-1}(P_j) \to P_j$ given by
$$\mu_{t, j} (z) = \mu_0(e^{- t \nu_j} z). $$
Further, the union $ \bigcup_{j = 1}^R e^{t \nu_j}
\mu_0^{-1}(P_j)$ is disjoint and therefore
$$\mu_t: \bigcup_{j = 1}^R e^{t \nu_j}
\mu_0^{-1}(P_j) \to P \backslash \ccal$$ is a diffeomorphism with
inverse $ \mu_{_t} ^{-1}(x) = e^{t \nu_j} \mu_0^{-1}(x)$.
\end{prop}

\begin{proof}

For $x \in P_j$, we have  \begin{equation}  \log \mu_{t}^{-1} (x)
= \nabla u_0(x) +   t \nabla f(x)) = \nabla u_0 + t \nu_j.
\end{equation}
Hence,
$$\begin{array}{lll}     \log \mu_{_t}^{-1} (x):  = \nabla
u_{_t}(x) & := & \nabla u_0 + t \nu_j \\ && \\
& = & \log \mu_0^{-1}(x) + t \nu_j \\ && \\
& = & \log (e^{t \nu_j} \mu_0^{-1}(x) \\ && \\
& \iff & \mu_{_t} ^{-1}(x) = e^{t \nu_j} \mu_0^{-1}(x).
\end{array}$$
It follows that $x = \mu_{_t}( e^{t \nu_j} \mu_0^{-1}(x))$ and
therefore $
   \mu_{_t}  (z) = \mu_0(e^{- t \nu_j} z)$
when $z \in e^{ t \nu_j} \mu_0^{-1}(P_j). $

It is easy to see that  images of the smooth regions under $\nabla
(u_0 + t f)$ are disjoint (see  Proposition \ref{ONEONE} for
details),   and therefore $\nabla (u_0 + t f)$ is a diffeomorphism
from the smooth chambers to their images, with inverse $\nabla
(\psi_t + \phip)$.

\end{proof}

We now give an analogous  formula on  the inverse image of the
subdifferential lying over the  corner set. To give the formula,
we introduce some notation. The corner set is a union of
hyperplanes of the form
\begin{equation} H_{j k} = \{x: \lambda_j (x) = \lambda_k(x) \} =
\{x: \langle \nu_j - \nu_k, x \rangle = v_k - v_j\}.
\end{equation}
The  inverse image of one hyperplane  under $\mu_0$ is the smooth
hypersurface of $M$ given by
\begin{equation}L_{i j}:=  \mu_0^{-1} (H_{j k}) = \{z \in M: \langle \nu_j -
\nu_k, \mu_0(z) \rangle = v_k - v_j\}. \end{equation}
\begin{prop} \label{MUTa} For each $t \geq 0$, $$\mu_t^{-1}(\ccal)
=  \bigcup_{j, k = 1}^R \; \bigcup_{\xi \in
 CH(\nu_j, \nu_k)} \;e^{t \xi} L_{j k}. $$
For $z \in \bigcup_{\xi \in
 CH(\nu_j, \nu_k)} \;e^{t \xi} L_{j k}, $ we have
$$ \mu_t^{j k } ( z) = \mu_0(\pi_t^{j k} (z)),  $$ where  $\mu^{jk}_t: \bigcup_{\xi \in
 CH(\nu_j, \nu_k)} \; e^{t \xi} L_{j
k}\to L_{j k}$ is the fibration  $e^{t \xi} w \to w$. Further,
$$\mu_t:\bigcup_{j, k = 1}^R \; \bigcup_{\xi \in
 CH(\nu_j, \nu_k)} \;e^{t \xi} L_{j k} \to \gcal(\partial (u + t f) )|_{\ccal}$$
is a homeomorphism whose inverse is defined for $(x, \nabla u_0(x)
+ t \xi) \in \gcal
\partial (u + t  f(x))$ by
$$\tilde{\mu}_{t} ^{-1}(x, \xi) = e^{t \xi} \mu_0^{-1}(x). $$
\end{prop}

\begin{proof}

By Corollaries \ref{LIFTEDMM} and \ref{LIFTEDMMa}, the  region
$\mu_t^{-1}(H_{j k})$ is parametrized by the map,
\begin{equation} \label{PARAM} M_t:  \;  (z, \xi) \in L_{j k}  \times  CH(\nu_j, \nu_k)
\to e^{t \xi} z,   \end{equation} where $e^{t \xi} \in \R_+^m$. By
(\ref{PARAM}), it follows that  $\mu_t^{-1}(H_{j k})$ fibers over
$L_{j k}$ with fibers given by the orbits of $e^{t \xi}$ with $\xi
\in CH(\nu_j, \nu_k)$. Then for  $z \in \mu_t^{-1}(H_{j k})$,
$\pi_t^{j k} (z) = \mu_0^{-1} (\mu_t(z)). $

In the inverse direction, for  $x \in \ccal$ and $\xi \in \partial
f(x)$, we have by definition of the lifted moment map
$\tilde{\mu}_t$,
\begin{equation} \iota(2 \log \tilde{\mu}_{t}^{-1} (x, \xi), x)  = (x,  \nabla u_0(x) +   t
\xi),
\end{equation}
or equivalently,
$$\begin{array}{lll}     2 \log \tilde{\mu}_{_t}^{-1} (x, \xi):
& = & \log \mu_0^{-1}(x) + t \xi \\ && \\
& = & \log (e^{t \xi} \mu_0^{-1}(x) \\ && \\
& \iff & \tilde{\mu}_{t} ^{-1}(x, \xi) = e^{t \xi} \mu_0^{-1}(x).
\end{array}$$
It follows that $x = \mu_{_t}( e^{t \nu_j} \mu_0^{-1}(x))$ and
therefore $
   \mu_{_t}  (z) = \mu_0(e^{- t \nu_j} z)$
when $z \in e^{ t \nu_j} \mu_0^{-1}(P_j). $

\end{proof}

We observe the analogy between  $e^{t \xi} \mu^{-1}(L_{jk})$ and
$e^{t \nu_j} \mu^{-1}(P_j)$, and  that the union of these domains
fills out $M$. The smooth and corner domains   meet along their
common boundary,
\begin{equation} \partial \left( \; \bigcup_{j = 1}^R e^{t \nu_j}
\mu_0^{-1}(P_j) \right) = \bigcup_{j, k = 1}^R \; \bigcup_{\xi \in
\partial CH(\nu_j, \nu_k)} \;e^{t \xi} L_{j k}. \end{equation}

Now that we have an explicit formula for $\mu_t$ we can give
simpler formulae for  $\psi_t$ than those of (\ref{PHIT}) and the
expressions following. They are stated in Theorem \ref{GEO}. For
clarity of exposition we restate them in the following

\begin{prop} \label{PHITFORM}  For any $z$,   $\psi_t(z) = F_t(\mu_t(z)) - I^z(\mu_t(z)).$
  Hence,

  \begin{itemize}

  \item (i) When $z = e^{\rho/2 + i\theta} \in \mu^{-1}(P_j)$, then
$\psi_t(\rho) = \phip( \rho - t \nu_j) - t v_j - \phip(\rho)$.

\item (ii) When $z \in \mu_0^{-1}(F^o)$, then $\psi_t(z) =
 \phi_F(\rho'' - t \nu_j'') - t v_j - \phi_F(\rho'')$.

 \item (iii) A point  $z = e^{\rho/2 + i \theta} \in \mu_0^{-1}(P_j \cap P_k)$ only if
 $\rho \in t CH(\nu_j, \nu_k)$. In that case,  $\mu_t(\rho) $ is a constant point $x_0$
 for $\rho \in t CH (\nu_j, \nu_k)$, and
$\psi_t(\rho) = \langle \rho, x_0 \rangle - u(x_0) - t (\langle
\nu_j, x_0 \rangle + v_j). $ Analogous formulae hold on the faces
of $\partial P$.

 \end{itemize}
\end{prop}

\begin{proof} In the  formula
 (\ref{PHIT}) for $\psi_t$,  it is now clear that the supremum is
 obtained at $x = \mu_t(z)$, giving the first formula.

 We can simplify the expression by substituting the expression for
 $I^z$. We illustrate with  the open orbit: In   the open real orbit,
 $\psi_t(\rho) + \phip(\rho) = \langle \rho, x_t(\rho) \rangle -  (u_0 + t
 f)(x_t(\rho))$ where $x_t(\rho)$ solves $\rho = \nabla (u_0 + t
 f)(x_t(\rho))$. On the subdomains $P_j$, $\nabla (u_0 + t f)$ is
 a map which inverts $\mu_t$. Hence, $\mu_t(\rho) = x_t(\rho)$.
 Since  $\mu_t(\rho) = \mu_0 (\rho - t \nu_j)$ and $f (\mu_0(\rho - t \nu_j))
 = \langle \nu_j, \mu_0(\rho - t \nu_j) \rangle - v_j $ in this region,
 $$\begin{array}{lll} \psi_t(\rho) + \phip(\rho) && = \langle \rho, \mu_t(\rho) \rangle - (u + t
 f)(\mu_t(\rho)) \\ && \\
 && = \langle \rho, \mu(\rho - t  \nu_j) \rangle - u
 (\mu(\rho - t \nu_j)) - tf (\mu(\rho - t \nu_j)) \\ && \\
 && = \langle \rho - t \nu_j, \mu(\rho - t \nu_j) \rangle - u
 (\mu(\rho - t \nu_j)) - tf (\mu(\rho - t \nu_j)) + t \langle \nu_j,
 \mu(\rho - t \nu_j) \rangle \\ && \\
 && = \phip( \rho - t \nu_j)  - t \left( f (\mu(\rho - t \nu_j)) - \langle \nu_j,
 \mu(\rho - t \nu_j) \rangle \right) \\ && \\
 && = \phip( \rho - t \nu_j) - t v_j .  \end{array}$$
Similarly, when  $z \in \mu^{-1}(F^o)$, then $\psi_t(z) +\phi_{F}=
 \phi_F(\rho'' - t \nu_j'')$.

 \end{proof}

\begin{rem}These expressions are consistent with $\psi_t(z) = F_t(\mu_t(z)) -
I^z(\mu_t(z)).$  For instance, on the open orbit,
\begin{equation}\begin{array}{lll}  \psi_t(z) && =  - t f(\mu_t(z)) -
u_0(\mu_t(z)) + \langle \mu_t(z), \log |z| \rangle - \phip(z)),
\end{array} \end{equation}
which is the same as (i) by the previous calculation.

\end{rem}

As a corollary, we have:
\begin{cor} \label{DEG} There exist open sets such that  $\omega_t^m \equiv
0$. \end{cor}

\begin{proof} The  sets $ \mu^{-1}(P_j \cap P_k)$ in (iii) of
Proposition \ref{PHITFORM} have non-empty interior. Indeed, they
are homeomorphic to the graph of $\partial (u + t f)$ along $P_j
\cap P_k$, and hence to the graph of $\partial f$ there. But
clearly the graph is a line-segment bundle over a hyperplane and
thus has full dimension.

\end{proof}

\subsubsection{Example}

 We work out the full formula in the case of $\CP^1,$ with $f(t) =
 |x - \frac{1}{2}|$. Thus, $\nu_1 = -1, v_1 = \frac{1}{2}; \nu_2 = 1, v_2 = - \frac{1}{2}$.
 The symplectic potential at time $t >
 0$ is $u_t(x) = x \log x + (1 - x) \log (1 - x) + t |x -
 \frac{1}{2}|. $
Then the subdifferential of $u_t$ is given by
$$ \partial u_t (x) = \left\{ \begin{array}{ll} \log
\frac{x}{1 - x} - t, & x < \frac{1}{2}, \\ & \\
\log \frac{x}{1 - x} + t (-1,1), & x = 0, \\ & \\
 \log \frac{x}{1
- x} + t, & x
> \frac{1}{2}.
\end{array} \right.
$$
The moment map on the open orbit  $\R$ is defined by
$$\mu_t(e^{\rho/2}) = x_t(\rho): \rho \in \partial
u_t (x_t(\rho)), $$ and $\mu_t(\rho) = \frac{1}{2}$ for $\rho \in
(-t, t)$. Since $\psi_t(\rho) = \langle \rho, \mu_t(\rho) \rangle
- (u + t f)(\mu_t(\rho)),$ we have
$$\psi_t(\rho) + \phip(\rho) =  \left\{ \begin{array}{ll} - \frac{t}{2} + \log (1 + e^{\rho + t}), & \rho \in (-\infty, -t)
 \\ & \\ \frac{\rho}{2} + \log 2 & \rho \in (-t, t)   \\ & \\
\frac{t}{2} +  \log (1 + e^{\rho - t}), &  \rho \in (t, \infty)
\end{array} \right. $$

\subsubsection{Formula for $\dot{\psi}_t$}

\begin{prop}\label{PHIDOT}  We have: $\dot{\psi}(t,z) = - f (\mu_t(z))$. \end{prop}

\begin{proof} We recall that on   the open real orbit,
 $\psi_t(\rho) = \langle \rho, x_t(\rho) \rangle -  (u_0 + t
 f)(x_t(\rho))$ where $x_t(\rho)$ solves $\rho = \nabla (u_0 + t
 f)(x_t(\rho))$. Hence,
$$\begin{array}{lll} \dot{\psi}_t(\rho) && =  - f(\mu_t(\rho))
+ \left(\langle \rho,  \frac{d}{dt} x_t(\rho) \rangle - \nabla (u
+ t
  f)(x_t(\rho)) \frac{d}{dt} x_t(\rho) \right), \end{array}$$
  and the parenthetical expression vanishes. There are analogous
  restricted expressions on $\mu^{-1}(F)$ for any boundary facet
  $F$, confirming that the identity holds for all $z \in M$.

\end{proof}

\subsection{\bf{$\psi_t \in C^{1,1}([0, L] \times M)$}}

In this section, we  complete the proof of the regularity
statement in Theorem \ref{GEO}. We do this in two steps to
separate  interior from boundary estimates: first we prove that
$\mu_t$ is Lipschitz and then we prove that $d\psi_t$ is
Lipschitz. The latter improves the former in terms of behavior
along $\dcal$.

\begin{prop} \label{MUISLIP} $\mu_t(z)$ is Lipschitz uniformly in $(t, z) \in [0,1]\times X$.

\end{prop}

\begin{proof} In view of Proposition \ref{LIP}, we only
 have to verify that $\mu_t$ is uniformly Lipschitz in $t$. Fix $t=t_0\geq 0$ and $z_0\in X$.

\begin{enumerate}

\item Suppose $z_0 \in (\mathbf{C}^*)^n$, $\rho_0  \in \mu_{t_0}^{-1} ( P_j^{\circ})$ and $x_t = \mu_t(z)$. Then $\mu_t$ is smooth in both $\rho$ and $t$ near $\rho_0$ and $t_0$. In fact,
    $\nabla (u + t f)(x_t) = \rho$ and
$D^2  (u_0 + t f)(x_t) \cdot \dot{x} = - \nabla f(x_t)$ after taking $t$-derivative. Therefore
     $$\frac{d}{dt} \mu_t(z) = -
\left(D^2 (u_0 + t f)(x_t) \right)^{-1} \nabla f(x_t) =
\left(D^2 (u_0 )(x_t) \right)^{-1} \nabla f(x_t)  $$
and so $\frac{d}{dt}\mu_t (z_0)$ at $t=t_0$ is uniformly bounded in $\mu_{t_0}^{-1}( \cup_j P_j^{\circ})$.

\item Suppose $z_0 \in (\mathbf{C}^*)^n$ and $\rho_0 \in \mathbf{R}^n \backslash \overline{ \mu_{t_0}^{-1} (\cup_j P_j^{\circ})}$. Then $\rho_0\in V_{x_{t_0}, t_0}$ and $\mu_t (z_0) = \mu_{t_0}$ for $t$ sufficiently close to $t_0$. Hence at $t=t_0$
    $$ \frac{d}{dt} \mu_t (z_0) = 0.$$

\item Suppose $z_0 \in (\mathbf{C}^*)^n$ and $\rho_0 \in \partial ( \mu_{t_0}^{-1} (\cup_j P_j^{\circ}))$. One sided derivatives of $\mu_t(z_0)$ exist and fall into the above cases at $t=t_0$.

\item Suppose $z_0 \in \mathcal{D}$. One only has to restrict $\mu_t$ on the subtoric variety and repeat the above argument.

\end{enumerate}

\end{proof}

We now complete the proof that $\psi_t \in C^{1,1}$ by proving
that
 $d \psi_t$ is Lipschitz.  Let us clarify first what
this adds to the statement that $\mu_t$ is Lipschitz:  In terms of
the basis $\{\frac{\partial}{\partial \theta_j}\}$ of $L \T$, we
have (cf. (\ref{MMDEF})) that
$$\mu_t (z) = (d (\phip + \psi_t)(\frac{\partial}{\partial
\rho_1}), \cdots, d (\phip + \psi_t)  ( \frac{\partial}{\partial
\rho_m})). $$ It  follows that $d \psi_t$ is continuous in
directions spanned by $\frac{\partial}{\partial \rho_j},
\frac{\partial}{\partial \theta_k}$. However, at $z_0 \in \dcal$,
some of these vector fields  vanish, namely those in the
infinitesimal isotropy group of $z_0$.  To show that $d \psi_t$ is
a Lipschitz one-form we need to show that $d \psi_t(X)$ is
Lipschitz when $X$ is a smooth non-vanishing vector field at $z_0
\in \dcal$. With no loss of generality, we may straighten out the
corner in which $\mu(z_0)$ lies and  hence that $\mu(z_0)$ lies in
a face of the corner at $0$ of the standard orthant. We use the
affine coordinates $z_j$ on $M$ adapted to $0$ and put $r_j =
|z_j|$. Then, $\frac{\partial}{\partial \rho_j} = \frac{1}{2} r_j
\frac{\partial}{\partial r_j}$ for $j = 1, \dots, m$. We assume
that $z_0$ lies in the facet of $\dcal$ defined by
 $r_1 = \cdots = r_k = 0$. Then,
 \begin{equation} d \psi_t (\frac{\partial}{\partial r_j}) =
 \frac{1}{r_j}  \mu_{t j} - d \phi_{P^o}
 (\frac{\partial}{\partial r_j}), \end{equation}
 where $\mu_{t j}$ is the $j$th component of $\mu_t$.
Thus, it suffices   to check that $\frac{1}{r_{\ell}} \mu_{t \ell}
(r)$ is Lipschitz at $r_1 = \cdots = r_k = 0$ for all $\ell = 1,
\dots, k$ , i.e. that   $\frac{1}{r_{\ell}} d \mu_{t {\ell}} (r)
\in L^{\infty}$. The Lipschitz property is thus equivalent to

\begin{enumerate}

\item  $\frac{1}{ r_{\ell}^2 } \mu_{t \ell} (r) \in L^{\infty}$,
and

\medskip

\item $ \frac{1}{r_{\ell}} d
 \mu_{t \ell} (r) \in L^{\infty}$ for each
$j$.

\end{enumerate}

\begin{prop} The 1-form $d\psi_t$ is Lipschitz continuous on $M$.
\end{prop}

\begin{proof}

We have explicit formulae for $\mu_t$ away from a codimension
subvariety of $M$. These formulae are sufficient by the following

\begin{lem}\label{HYPER}  Let $X$ be a compact K\"ahler manifold. Let $Y$ be
 a subvariety of $X$ and $g$ a fixed K\"ahler metric on $X$. If
  $\mu \in C^0(X)$ and $|d \mu|_g \leq C$ on $X\backslash Y$ for a uniform constant $C>0$,
  then $\mu$ is Lipschitz on $X$.

\end{lem}

\begin{proof} Let $z_1$, $z_2$ be two arbitrary points on
$X$. Let $z_{1,\epsilon}\in B_{\epsilon}(z_1)\cap (X\backslash Y)$
and $z_{2,\epsilon}\in B_{\epsilon}(z_2) \cap (X\backslash Y)$. Let
$\gamma_{\epsilon}$ be the shortest geodesic joining $z_{1, \epsilon}$
and $z_{2, \epsilon}$ on $X$. Without loss of generality, we can assume
 that $\gamma_{\epsilon}$ only intersects with $Y$ by finite points,
 say, $w_1$, ... , $w_k$. Let $w_0= z_{1, \epsilon}$ and $w_{k+1} = z_{2, \epsilon}$. Then

\begin{equation}
 |\mu(z_{1, \epsilon}) - \mu (z_{2, \epsilon}) | \leq  \sum_{j=0}^{k} |\mu(w_j) - \mu(w_{j+1})|
 \leq  C \sum_j | w_j - w_{j+1}|_g \leq  C|\gamma_{\epsilon}|_g.
\end{equation}

\end{proof}
Thus it suffices to have uniform bounds for (1) and (2) on the
open sets where we have explicit formulae. We first prove uniform
bounds on the smooth domains.

\begin{lem} \label{SMOOTH} With the above notation,  $ d  \mu_{t\ell} $ satisfies the bounds (1)-(2) on
the sets  $ e^{t \nu_j} \mu_0^{-1}(P_j)$ for each $j$. \end{lem}

\begin{proof}

  By
Proposition \ref{MUT}, on the set $ e^{t \nu_j} \mu_0^{-1}(P_j)$
we have $\mu_{t, j} (z) = \mu_0(e^{- t \nu_j} z) $.  Properties
(1) and (2) above follow immediately because any smooth moment map
$\mu_0$ satisfies these estimates, as may be seen from   the fact
that $\nabla^2 u_0$ has first order poles on the boundary facets.

\end{proof}

We now prove the bounds in the   the complementary  set   $ M
\backslash \left(\bigcup_{j = 1}^R e^{t \nu_j}
\mu_0^{-1}(P_j)\right) $. By Proposition \ref{MUTa} it suffices to
verify the bounds on each set $e^{t \xi} L_{j k}$ with $\xi \in
CH(\nu_j, \nu_k)$.

\begin{lem} \label{CORN}  for each $j, k$ $D \mu_t$ satisfies the bounds (1)-(2)
in each set $U_{j k}: = \{e^{t \xi} L_{j k}, \;\; \xi \in
CH(\nu_j, \nu_k)\}$.
\end{lem}

\begin{proof}


By Proposition \ref{MUTa}, we have $\mu_t(z) = \mu_0(\pi_t(z))$
where $\pi_t$ is the fiber map from $U_{j k} \to L_{j k}$ with
fibers the orbits of $e^{t \xi}$. Hence, for each $\ell$,
\begin{equation} \label{ONE}  \frac{1}{r_{\ell}(z)} d \mu_t (z) =
\frac{1}{r_{\ell}(z)} (d \mu_{0 \ell}) (\pi_t(z)) \circ D
\pi_t(z). \end{equation}  We now observe that for any compact set
$K$ of $\nu \in \R^m$, there exists $C_K \in \R_+$ so that
\begin{equation} \label{CK}  \frac{r_j(e^{\nu} z)}{r_j(z)} \leq C_K.
\end{equation}
Indeed, if we write $z = e^{\rho/2}$ then  $r_j(z) =
e^{\frac{\rho_j}{2}}$ and $\frac{r_j(e^{\nu} z)}{r_j(z)} =
e^{\nu_j}.$ It follows that
\begin{equation} \label{TWO} \begin{array}{lll}  \| \frac{1}{r_{\ell}(z)} d \mu_t (z)\|
& \leq & C_k \;  \frac{1}{r_{\ell}(\pi_t(z))} \|(d \mu_{0 \ell})
(\pi_t(z)) \circ D \pi_t(z) \|. \end{array} \end{equation} To
complete the proof of the Lemma, we need to show that

\begin{itemize}

\item $\frac{1}{r_{\ell}(w))} d \mu_{0 \ell} (w) \in
L^{\infty}(L_{j k});$
\medskip

\item $D \pi_t(z) \in L^{\infty}(U_{j k}). $
\end{itemize}

The first statement holds in a neighborhood of the facet $r_{\ell}
= 0$ of $\dcal$ and hence holds on its intersection with  $L_{j
k}$; the statement reduces to Lemma \ref{MUTLIP} away from this
open set.

Hence the key issue is to prove the boundedness of $D \pi_t$ in
$U_{j k}$. We first note that  $L_{j k}$ is the  level set $I_{j
k} = v_k - v_j$ of the function \begin{equation} \label{IJK} I_{j
k} (z) : = \langle \nu_j - \nu_k, \mu_0(z) \rangle. \end{equation}
Furthermore the gradient flow of this function with respect to
$\omega_0$ is given by the subgroup $\sigma \to e^{\sigma (\nu_j -
\nu_k)}$ of the $\R_+^m$ action. Indeed, the latter is the joint
action of the gradient flows $\nabla I_j$  of the action variables
$I_j$ which form the components of the moment map $\mu_0: M \to P,
$ i.e. $\mu_0 = (I_1, \dots, I_m)$. This holds because the
Hamilton vector field $H_{I_j}$ with respect to $\omega_0$ equals
$\frac{\partial}{\partial \theta_j}$ and its image under the
complex structure $J$ equals both $\frac{\partial}{\partial
\rho_j}$ and $\nabla I_j$ where $\nabla$ is the gradient for the
metric $g_{\omega}(X, Y) = \omega(J X, Y)$. Thus,  it follows
directly from (\ref{IJK}) that the gradient flow of $I_{j k}$ is
$\sigma \to e^{\sigma (\nu_j - \nu_k)}$.

Now $\xi \in CH(\nu_j, \nu_k)$  has the form $\xi = \nu_k + s
(\nu_j - \nu_k)$ and so $e^{t \xi}  = e^{t \nu_k} e^{t s (\nu_j -
\nu_k)}. $  Thus, the family of hypersurfaces $e^{t \xi} L_{j k}$
for $\xi \in  CH(\nu_j, \nu_k)$  is the image under $e^{t \nu_k}$
of the family of hypersurfaces  $e^{t s (\nu_j - \nu_k)} L_{j k}$,
and the latter is a family of level sets of $I_{j k}$. In
particular, for each $t$ and $s$, the latter family is orthogonal to the flow lines of $e^{ts(\nu_j -\nu_k)}$ at $e^{ts(\nu_j-\nu_k)}z \in M$ with respect to the K\"ahler metric $(e^{ts(\nu_j-\nu_k)})^* g_{\omega}(z) = g_{\omega} ( e^{ts(\nu_j-\nu_k)/2} z )$. This is because $e^{ts(\nu_j -\nu_k)} L_{jk}$ is the level set given by $$ I_{jk}(z; t, s) = \langle \nu_j - \nu_k, \mu_0 ( e^{ts(\nu_j - \nu_k)/2} z) \rangle.$$

Moreover, $g_{\omega} ( e^{ts(\nu_j-\nu_k)} \cdot )$ is equivalent to $g_{\omega} (  \cdot )$ and it can  be seen in $\mathbf{R}^n$. Suppose $$ I_{jk}(\rho; t, s) = \langle \nu_j - \nu_k, \mu_0 ( \rho+ ts (\nu_j -\nu_k) ) \rangle.$$ Hence

$$ \nabla_{\rho} I_{jk}( \rho; t, s) =\nabla^2_{\rho} \varphi_0 (\rho+ ts(\nu_j -\nu_k))\cdot(\nu_j - \nu_k) $$
and so under the Riemannian metric $$G_{ts} (\rho) = G_0 (\rho+ ts(\nu_j -\nu_k)) = \sum_{p,q}\frac{\partial^2 \varphi_0}{\partial \rho_p \partial \rho_q} \varphi_0 (\rho+ts(\nu_j -\nu_k)) d\rho_p \otimes d\rho_q$$ on $\mathbf{R}^n$, and $\nu_j -\nu_k$ is orthogonal to the hypersurface. It is obvious that $G_{ts}$ and $G_0$ are uniformly equivalent.

It is convenient to slightly modify our problem by removing $e^{t
\nu_k}$. In  the notation of (\ref{PARAM}), we  change $M_t$ to
$\tilde{M}_t: L_{jk}  \times [0, 1] \to M,$ where $ \;
\tilde{M}_t(z, s) : = e^{ st (\nu_j - \nu_k)/2} z. $ Thus, $M_t (z,
\nu_k + s (\nu_j - \nu_k)) = e^{\nu_k t} \tilde{M}_t(z, s). $ We
then define $\tilde{\pi}_t: e^{- t \nu_l} U_{j k}  \to L_{jk}$ by
$\tilde{\pi}_t \tilde{M}_t(z, s) = z. $ Since $\pi_t(w) =
\tilde{\pi}_t(e^{- \nu_k t/2} w)$,  $D \pi_t $ is bounded on $U_{j
k}$ if and only if $D \tilde{\pi}_t$ is bounded on $e^{- t \nu_k}
U_{jk}$.

To prove that $D \tilde{\pi}_t$ is bounded on $e^{- t \nu_k}
U_{jk}$, we observe that the gradient flow of $I_{j k}$ at fixed
$s, t$  takes $L_{j k}$ to another level set of $I_{j k}$ and
hence one has orthogonal foliations of $e^{- t \nu_k} U_{jk}$
given by level sets and gradient lines of $I_{j k}$. We note that
the  critical points of $I_{j k}$ occur only  on $\dcal$ and by
Lemma \ref{HYPER} it suffices to bound $D \tilde{\pi}_t$ on its
complement. We may thus  split the tangent space at each point of
$ e^{- t \nu_k} U_{jk}$ into $\R \nabla I_{j k} \oplus T \{I_{j k}
= C\}$. Since $D \tilde{\pi}_t = 0$ on $\R \nabla I_{j k}$ it then
suffices to bound $D \left( \tilde{\pi}_t |_{\{I_{j k} = C\}}
\right) $ uniformly in $C$ as $C$ runs over the levels in $e^{- t
\nu_k} U_{jk}$. But $\tilde{\pi}_t |_{\{I_{j k} = C\}}$ is simply
the inverse of the map  $z \in L_{j k}  \to e^{ t s(C) (\nu_j -
\nu_k)}$ where $t s(C)$ is the parameter time of the gradient flow
from $I_{j k } = v_k - v_j$ to the level set $I_{j k} = C$. Hence
$D \tilde{\pi}_t$ is bounded above as long as the  derivatives of
the family of maps $z \to  e^{ t s(C) (\nu_j - \nu_k)/2} z$ on
$L_{jk}$ have a uniform lower bound. But this is clear since this
family forms a compact subset of the group $\R_+^m$.

This concludes the proof of the Lemma.

\end{proof}

Lemmas \ref{SMOOTH} and \ref{CORN} imply that $d\psi_t$ is
Liptschitz, concluding the proof of the Proposition.

\end{proof}

Finally, we consider $t$ derivatives:

\begin{prop} $\dot{\psi_t}$ is Lipschitz on  $([0,1] \times X ).$

\end{prop}

\begin{proof}

By Proposition \ref{PHIDOT}, and the fact that
 both $f$ and $\mu_t$ are continuous, it follows that  $\psi_t \in C^1( [0,
T]\times X)$. Obviously, $\dot{\psi_t}$ is smooth outside a
subvariety of $[0, T]\times X$ and so it suffices to check the
uniform Lipschitz condition for $f (\mu_t(\rho))$ in both $z$ and
$t$ variables. In the $z$ variables it follows from Proposition
\ref{MUISLIP} and the fact that  $f$ is  Lipschitz.

In the $t$ variable, we note that $\frac{\partial}{\partial t}
f(\mu_t(\rho)) = \nu_i \cdot \frac{\partial}{\partial t} \mu_t
(\rho)$, and this is bounded by Proposition \ref{MUISLIP}.
\end{proof}

\begin{rem}

We note that the geodesic equation
$$\partial_t^2 {\phi}_t = |\partial_z \dot{\phi}_t|^2_{\omega_t} $$
is valid in a weak sense,  although both sides are discontinuous,
hence $\psi_t$  is  a weak solution of the geodesic equation, or
equivalently of the Monge-Amp\`ere equation $(\ddbar \Phi)^{m+1} =
0$ where $\Phi(t + i \tau, z) = \psi_t(z)$ (cf. \cite{S,D1} for
the relation of the geodesic equation and the Monge-Amp\`ere
equation). Since the Monge-Amp\`ere measure is $\T$-invariant, it
is equivalent that the real Monge-Amp\`ere measure of $\psi_t$ on
$\R \times M_{\R}$ equals zero.   In  the real domain, a weak
solution of the Monge-Amp\`ere equation is  a function $\psi$
whose Monge-Amp\`ere measure $M(\psi)$ equals zero, where the
Monge-Amp\`ere measure is defined by $M(\psi)(E) = \left|
\partial \psi(E) \right|$, i.e. by the Lebesgue measure of the
image of a Borel set $E$ under the subdifferential map of $\psi$
(see e.g. \cite{CY}).

To see that our $\psi_t$ solves the homogeneous real
Monge-Amp\`ere equation, we note that  the image of the gradient
map of $\psi_t$ is the same as the image of the subdifferential
map (in both the $t$ and $x$ variables) of $u + t f$. Since the
latter is linear in $t$, its Monge-Amp\`ere measure in $\R \times
P$ equals zero. We conclude that $(\partial_t^2 {\phi}_t -
|\partial_z \dot{\phi}_t|^2_{\omega_t}) dt dx $ is the zero
measure. It follows that, as measures, $\partial_t^2 {\phi}_t dt
dx = |\partial_z \dot{\phi}_t|^2_{\omega_t} dt dx$.

It follows that  $\partial_t^2 \psi_t \in L^{\infty}$ if and only
if $|\partial_z \dot{\phi}_t|^2_{\omega_t}) \in L^{\infty}$. The
metric norm uses the inverse of $\omega_t$, which as observed
above vanishes on the open sets $\mu(P_i \cap P_j)$. On the other
hand, the formula in Theorem \ref{GEO} shows that $\partial_z
\dot{\phi}_t \equiv 0$ there as well.

\end{rem}

\section{\label{PFC1} $C^1$ convergence: Proof of Proposition \ref{LLNMUZT}}

In this section we prove that $\psi_k(t, z) \to \psi_t$ in $C^1$.
The proof uses the properties of the moment map $\mu_t$
established in previous sections,  and is based on a strong
version of Varadhan's Lemma and on uniform large deviations. It
follows that $\psi_t \in C^1([0, L] \times M)$ as will be
discussed at the end.

\subsection{Uniform tilted large deviations upper bound}

Our aim is  to show: \begin{prop} \label{LDT} For compact sets,
\begin{equation} \frac{1}{k} \log \mu_k^{z,t} (K) \leq - \inf_{x \in K} I^{z, t}(x) + o(1), \end{equation}
where $o(1)$  denotes a quantity such that $o(1) \to 0$ as $k \to
\infty$ uniformly in $z, t$ when $t$ lies in a compact set. Here,
$I^{z,t}(x) = [I^z(x) - F_t(x)] - \sup_{x \in X} [F_t(x) -
I^z(x)]$.
\end{prop}

\begin{proof}

  Since the logarithmic asymptotics of the denominator
$Z_k^{z,t}$ in $\mu_k^{z,t}$ follow by the Varadhan's Lemma just
proved, it suffices to show that
\begin{equation} \frac{1}{k} \log \int_K  e^{k F_{t} (x)}
d\mu_k^z(x) \leq - \inf_{x \in K} [F_t(x) - I^z(x)] + o(1), \;\; k
\to \infty,
\end{equation}
where the remainder is uniform in $z \in M$ and $t \in [0, L]$ for
any $L > 0$. We prove this with a slight generalization   of Lemma
\ref{IZLOGPCAL}. The proof is essentially the same, but for the
reader's convenience we include some details.

\begin{lem} \label{IZLOGPCALt}  Let  $L_k(z, t,  \frac{\alpha}{k}) = \frac{1}{k} \log
e^{k F_t(\frac{\alpha}{k})}
\frac{|s_{\alpha}(z)|^2_{h^k}}{Q_{h^k}(\alpha)} - (F_t - I_z)
(\frac{\alpha}{k})$. Then $L_k(z, t,  \frac{\alpha}{k}) =
 - \frac{1}{k} \log Q_{h^k}(\alpha) +
u(\frac{\alpha}{k})$ for $z \in M^o$ and  satisfies
$$L_k(z, t, \frac{\alpha}{k}) = O(\frac{1}{k}), \;\; (k \to \infty) $$
uniformly in $z \in M^o$, $t \in [0, L]$ and $\alpha \in k P$. The
same formula and remainder hold when $\mu_t(z) \in F$ and
$\frac{\alpha}{k} \in F$.

\end{lem}

\begin{proof}

First assume that  $z = e^{\rho/2} \in M^o$. Then,
\begin{equation} \label{ZINT}\begin{array}{lll} \frac{1}{k} \log
e^{k F_t(\frac{\alpha}{k})}
\frac{|s_{\alpha}(z)|^2_{h^k}}{Q_{h^k}(\alpha)} &= &
F_t(\frac{\alpha}{k}) +  \langle \frac{\alpha}{k}, \rho \rangle -
\phi(\rho) - \frac{1}{k} \log Q_{h^k}(\alpha),\end{array}
\end{equation}
while \begin{equation} \label{IZALPHA} (F_t - I^{z})
(\frac{\alpha}{k}) = F_t(\frac{\alpha}{k}) - \langle
\frac{\alpha}{k}, \rho \rangle + u(\frac{\alpha}{k}) +
\phi(\rho).\end{equation}  Hence,
\begin{equation} \label{IZALPHAa}\frac{1}{k} \log
\frac{|s_{\alpha}(z)|^2_{h^k}}{Q_{h^k}(\alpha)} - (F_t -
I^z(\frac{\alpha}{k})) = - \frac{1}{k} \log Q_{h^k}(\alpha) +
u(\frac{\alpha}{k}). \end{equation}  The rest continues as in the
case $t = 0$.
 \end{proof}

The following Lemma, a generalization of Proposition \ref{UNIFUB},
concludes the proof:

\begin{lem}\label{UNIFUBt}  For any compact subset $K \subset \bar{P}$, we have the uniform  upper bound
$$\frac{1}{k} \log \mu_k^{z,t} (K) \leq -  \inf_{x \in K}
I^{z,t} (x) +   O(\frac{\log k}{k}), $$ where the remainder is
uniform in $z$ and $t \in [0, L]$.

\end{lem}

\begin{proof} Taking into account the denominator in $\mu_k^{z,t}$
and Lemma \ref{IZLOGPCALt}, the proof as in Lemma \ref{UNIFUB}
leads to the conclusion,
$$\frac{1}{k} \log \mu_k^{z,t}(K) = \frac{1}{k} \log \sum_{\alpha \in k P: \frac{\alpha}{k} \in K}
e^{- k I^{z,t}(\frac{\alpha}{k})}  + O(\frac{1}{k}) + O(\frac{\log
k}{k}) $$ and the rest of the proof goes as before.
\end{proof}

\end{proof}

 \subsection{Proof of Proposition \ref{LLNMUZT}}

 We first prove:

 \begin{lem}\label{PROPLEM}  $\mu_t(z) = \lim_{k \to \infty} x d\mu_k^{z,t}$ and
 $\mu_k^{z,t} \to \delta_{\mu_t(z)}.$
 \end{lem}

\begin{proof} We first use   \cite{E}, Theorem II.6.3
(particularly the proof that $\E(W_n/a_n) \to \nabla c(0)$ on page
49) to show that
  $\mu_k^{z,t} \to
\delta_{\mu_t(z)}$ in the weak sense as $k \to \infty.$

The proof uses the logarithmic moment generating function for
$d\mu_k^{z,t}$, defined  as before by
$$\begin{array}{lll} \Lambda^{z, t}(\rho) & = &  \lim_{k \to \infty} \frac{1}{k} \log
\int_P e^{k \langle x, \rho \rangle} d\mu_k^{z,t} \\ && \\
& = &  \lim_{k \to \infty} \frac{1}{k} \log \int_P e^{k ( \langle
x, \rho \rangle + F_t(x) } d\mu_k^{z} - \frac{1}{k} \log
Z_k^{z,t}.
\end{array} $$

By Varadhan's Lemma, the first term tends to
$$\sup_{x \in P} \left(\langle x, \rho \rangle + F_t(x)  - I^z(x) \right) $$
and the second tends to
$$\sup_{x \in P} \left( F_t(x)  - I^z(x) \right).  $$
 Thus, we have
\begin{equation} \Lambda^{z,t}(\rho)= \sup_{x \in P} \left(\langle x, \rho \rangle + F_t(x) - I^z(x) \right)
- \sup_{x \in P} \left( F_t(x) - I^z(x) \right).
\end{equation}
Up to the constant $R t$, the first term defines the Legendre
transform of the strictly convex function $I^z + t f$ and since
the second term is constant in $\rho$, $\Lambda^{z,t}(\rho)$  is a
strictly convex function of $\rho$, which up to a constant equals
\begin{equation} \label{LAMBDAZT} \left\{ \begin{array}{ll}
\lcal(u_0 + t f) (\rho + \log |z|) = \psi_t(\rho + \log |z|) +
\phip(\log |z|) , & z
\in M^o; \\ &\\
\lcal_F(u_0 + t f)(\rho'' + \log |z''|) = \psi_t(\rho'' + \log
|z''|) + \phi_F(|z''|), & z \in M_F. \end{array} \right.
\end{equation}
In the evaluation on $F$ we note that the supremum is taken for $x
\in F$ and hence $\langle x, \rho > = \langle x'', \rho'' \rangle$
where $\rho''$ is the component of $\rho$ along $F$.

By Proposition \ref{U+TFC1}, $\Lambda^{z,t}(\rho)$ is a $C^1$
function of $\rho \in \R^m$ for all $z \in M^o$ and it is a $C^1$
function of $\rho'' \in \R^{m -k}$ for $z \in M_F$.  Then by
\cite{E}, Theorem II.6.3 (particularly the proof that $\E(W_n/a_n)
\to \nabla c(0)$ on page 49) it follows that
$$ \lim_{k \to \infty} \int_P x d\mu_k^{z,t} \to  \left\{ \begin{array}{ll} \nabla (\psi_t + \phip)(
\log |z|), & z \in M^o; \\ & \\
\nabla'' (\psi_t + \phip) ( \log |z''|), & z \in M_F. \end{array}
\right.
$$
By Definition (\ref{MUTDEF}), the limit equals  $\mu_t(z)$ in all
cases, and Theorem II. 6. 3 of \cite{E} shows that $\mu_k^{t, z}
\to \delta_{\mu_t(z)}$.

\end{proof}

\subsection{$C^1$ Convergence and Varadhan's Lemma}

We now use the  large deviations principle of the previous section
to prove that the limit in Lemma \ref{PROPLEM} is $C^0$ and at the
same to oomplete the proof of Proposition \ref{LLNMUZT}.

\begin{proof} Each first derivative has the form
$$\int_{\bar{P}} \psi(x) \; d\mu_k^{t, z}(x), $$
where $d\mu_k^{t, z}$ is the time-tilted measure (\ref{TIMETILT})
and $\psi \in C(\bar{P})$. Indeed, on the open orbit,

\begin{eqnarray*}
\frac{1}{k}\nabla_{\rho} \log \sum_{\alpha \in k P \cap \Z^m} e^{2
t k  ( R - f(\frac{\alpha}{k }))} \frac{|s_{\alpha}(z)|^2_{h
^k}}{Q_{h^k}(\alpha)} &=& \frac{\sum_{\alpha \in k P \cap \Z^m
}\left( \frac{\alpha}{k} - \mu(z) \right) e^{2 t k  ( R -
f(\frac{\alpha}{k }))} \frac{|s_{\alpha}|^2_{h^k}
}{\Q_{h^k}(\alpha)} } {\sum_{\alpha \in k P \cap \Z^m } e^{2 t k
( R - f(\frac{\alpha}{k }))}
\frac{|s_{\alpha}|^2_{h^k} }{Q_{h^k}(\alpha)} }\\ \\
& = & \int \left( x - \mu(z) \right) d\mu_k^{z,t}(x)
\end{eqnarray*}
Near $\dcal$ we cannot express derivatives in terms of
$\nabla_{\rho}$ but must rather use $\nabla_z$ as in \cite{SoZ2}.
This only has the effect of changing  the sum over  $\alpha$ to a
sum over $\alpha_n \not= 0$ and then changing $\alpha \to \alpha -
(0, \dots, 1_n, \dots, 0)$.

Also, the  $\partial_t$ derivative has the form
\begin{equation} \begin{array}{lll}
\frac{1}{k} \frac{\partial} {\partial t}\log \sum_{\alpha \in k P
\cap \Z^m} e^{2 t k  ( R - f(\frac{\alpha}{k }))}
\frac{|S_{\alpha}(z)|^2_{h^k}}{Q_{h^k}(\alpha)} & = & \frac{1}{k}
\frac{\sum_{\alpha} e^{2 t k  ( R - f(\frac{\alpha}{k }))}
\left(2 k (R - f(\frac{\alpha}{k}) \right)
\frac{e^{k\langle\alpha, \rho
\rangle}}{Q^k(\alpha)}}{\left(\sum_{\alpha} e^{2 t k  ( R -
f(\frac{\alpha}{k }))}  \frac{e^{k\langle\alpha, \rho
\rangle}}{\QQ_{h_t^k}(\alpha)}\right)}\\ && \\
& = & \int_P \left(2  (R - f(x) \right) d\mu_k^{z,t}(x).
 \end{array}
\end{equation}

Hence to prove the result we need only to show:
\begin{lem}\label{PSI}  For any continuous $\psi \in C(\bar{P})$,
$\int_{\bar{P}} \psi(x) \; d\mu_k^{t, z}(x) = \psi(\mu_t(z)) +
o(1), $ with uniform remainder. \end{lem}

To prove the Lemma we first note that
 $I^{z,t}(x)$ attains
its infimum at a unique point $\bar{x}(z,t) = \mu_t(z)$ where
$\mu_t$ is the moment map of $\psi_t$.

We have,
$$\int_{\bar{P}} \psi(x) \; d\mu_k^{t, z}(x) = \int_{|x - \mu_t(z)| \leq \epsilon}
 \psi(x) \; d\mu_k^{t, z}(x) + \int_{|x - \mu_t(z)| >  \epsilon}
 \psi(x) \; d\mu_k^{t, z}(x). $$
 By Proposition \ref{LDT}, the  second term is bounded by
$$\begin{array}{lll} \max_{P} | \psi| \cdot \; \mu_k^{z, t} (|x - \mu_t(z)| >
\epsilon) & \leq  & \max_{P} | \psi| \cdot \; e^{- k \left(
\inf_{|x - \mu_t(z)|
> \epsilon} I^{z, t}(x) + o(1) \right)}\\ && \\
& = & o(1), \end{array}$$ where the remainder $o(1)$ is uniform
for any $\epsilon > 0$. Since $\inf_{z \in M, t \in [0, L]}
\inf_{|x - \mu_t(z)|
> \epsilon} I^{z, t}(x) > 0$, the final estimate is uniform in $t,
z$.

We then consider the first term.  Since $\psi$ is uniformly
continuous on $\bar{P}$, there exists $\epsilon $ for any given
$\delta > 0$ so that $|\psi(x) - \psi(\mu_t(z))| \leq \delta$ if
$|\mu_t(z) - x| \leq \epsilon$. The first term is then
$\psi(\mu_t(z))$ plus $O(\delta)$. Choosing $\delta$ sufficiently
small and then $k$ sufficiently large completes the proof of the
Lemma.
\end{proof}

Note that this gives another proof that $\psi_t \in C^1([0, L]
\times M)$.

\end{document}